\documentclass[12pt]{amsart}
\usepackage{amsmath,amssymb,amsthm,amscd,amsxtra,eucal,latexsym,mathrsfs}

\newtheorem{cor}{Corollary}
\newtheorem{lem}{Lemma}
\newtheorem{prop}{Proposition}

\newtheorem{conj}{Conjecture}
\newtheorem{thm}{Theorem}

\theoremstyle{remark}
\newtheorem{Rem}{Remark}

\DeclareMathOperator {\GL}{GL}
\DeclareMathOperator {\Fl}{Fl}

\DeclareMathOperator {\id}{id}

\DeclareMathOperator {\End}{End}

\newcommand {\eql}[2]{\begin{equation}\label{#1}#2\end{equation}}

\newcommand{\BZ}{{\mathbb Z}}
\newcommand{\BN}{{\mathbb N}}

\newcommand {\gS}{\mathfrak S}

\newcommand {\bH}{\mathbf H}
\newcommand {\bR}{\mathbf R}

\newcommand {\N}{\mathcal N}
  
\newcommand{\CR}{{\mathcal R}}
\newcommand{\CH}{{\mathcal H}}  

\newcommand{\sk}{{\mathsf k}}

\renewcommand{\b}{{\beta}}

\newcommand{\s}{\sigma}

\newcommand{\tw}{{\tilde w}}

\newcommand{\btu}{\bigtriangleup}
\newcommand{\invl}{\leftrightarrow}

\newcommand{\nc}{\newcommand}
\nc{\on}{\operatorname}
\nc{\ol}{\overline}
\newcommand{\Ql}{{{\underline{\overline{\mathbb Q}}}{}_l}}
\nc{\Fq}{{\mathbb F}_q}
\nc{\Fqb}{\ol{{\mathbb F}_q}}

\newcommand{\utH}{\tilde{\underline H}{}}
\nc{\aff}{{\on{aff}}}
\nc{\sph}{{\on{sph}}}
\nc{\Tate}{{\on{Tate}}}
\nc{\Hall}{{\on{Hall}}}
\nc{\bHall}{{\mathbf{Hall}}}
\nc{\CHall}{{{\mathcal H}a\ell\ell}}
\nc{\Mall}{{\on{Mall}}}
\nc{\bMall}{{\mathbf{Mall}}}
\nc{\CMall}{{{\mathcal M}a\ell\ell}}
\nc{\MA}{{\mathcal{MA}}}
\nc{\MC}{{\mathcal{MC}}}
\nc{\Char}{{\mathsf{CH}}}

\nc{\hR}{\hat{\mathcal R}{}^\aff}
\nc{\hRs}{\hat{\mathcal R}{}^\sph}
\nc{\tR}{\tilde{\mathcal R}{}^\sph}
\nc{\tY}{{\tilde{\mathfrak Y}_{n,m}}}
\nc{\fY}{{{\mathfrak Y}_{n,m}}}
\nc{\tX}{{\tilde{\mathfrak X}_{n,m}}}
\nc{\fX}{{{\mathfrak X}_{n,m}}}

\nc{\Gm}{{{\mathbb G}_m}}
\nc{\Pinf}{{\mathbf P}}
\nc{\Flaff}{{\mathbf{Fl}}}
\nc{\Gr}{{\mathbf{Gr}}}
\nc{\bV}{{\mathbf V}}
\nc{\bVo}{{\overset{\circ}{\mathbf V}}}
\nc{\bP}{{\mathbf P}}
\nc{\bO}{{\mathbf O}}
\nc{\bF}{{\mathbf F}}
\nc{\bI}{{\mathbf I}}
\nc{\bv}{{\mathbf v}}
\nc{\bq}{{\mathbf q}}
\nc{\blambda}{{\boldsymbol{\lambda}}}
\nc{\bmu}{{\boldsymbol{\mu}}}
\nc{\bnu}{{\boldsymbol{\nu}}}
\nc{\bpi}{{\boldsymbol{\pi}}}
\nc{\bfeta}{{\boldsymbol{\eta}}}
\nc{\bzero}{{\boldsymbol{0}}}
\nc{\GO}{{\mathbf{G}_\bO}}
\nc{\GF}{{\mathbf{G}_\bF}}
\nc{\iso}{{\stackrel{\sim}{\longrightarrow}}}
\nc{\tB}{{\widetilde{\mathcal B}}}

\nc{\BO}{{\mathbb O}}
\nc{\CA}{{\mathcal A}}
\nc{\CB}{{\mathcal B}}
\nc{\calC}{{\mathcal C}}
\nc{\CY}{{\mathcal Y}}
\nc{\CF}{{\mathcal F}}
\nc{\CG}{{\mathcal G}}
\nc{\CM}{{\mathcal M}}
\nc{\CN}{{\mathcal N}}
\nc{\CP}{{\mathcal P}}
\nc{\CS}{{\mathcal S}}
\nc{\CU}{{\mathcal U}}
\nc{\CW}{{\mathcal W}}
\nc{\CX}{{\mathcal X}}
\nc{\fc}{{\mathfrak c}}
\nc{\fu}{{\mathfrak u}}
\nc{\fv}{{\mathfrak v}}

\begin{document}

\title{Mirabolic affine Grassmannian and character sheaves}
\author{Michael Finkelberg} 
\address{IMU, IITP and State University Higher School of Economy\\
Mathematics Department, rm. 517\\
20 Myasnitskaya st.\\
Moscow 101000 Russia}
\email{fnklberg@gmail.com}
\author{Victor Ginzburg}
\address{University of Chicago, 
Mathematics Department, Chicago, 
IL 60637, USA}
\email{ginzburg@math.uchicago.edu}
\author{Roman Travkin}
\address{Massachusetts Institute of Technology\\
Mathematics Department, Cambridge MA 02139 USA}
\email{travkin@math.mit.edu}

\begin{abstract}
We compute the Frobenius trace functions of mirabolic character
sheaves defined over a finite field. The answer is given in terms of
the character values of general linear groups over the finite field,
and the structure constants of multiplication in the mirabolic
Hall-Littlewood basis of symmetric functions, introduced by Shoji.
\end{abstract}

\maketitle

\bigskip

\section{Introduction}

\subsection{}
This note is a sequel to~\cite{T}. We make a free use of notations and
results thereof. Our goal is to study the mirabolic character sheaves
introduced in~\cite{FG}. According to Lusztig's results, the unipotent
character sheaves on $\GL_N$ are numbered by the set of partitions of
$N$. For such a partition $\lambda$ we denote by $\CF_\lambda$ the
corresponding character sheaf. If the base field is $\sk=\Fq$,
the Frobenius trace function of a character sheaf
${\mathcal F}_\lambda$ on a unipotent class of type $\mu$ is 
$q^{n(\mu)}K_{\lambda,\mu}(q^{-1})$ where $K_{\lambda,\mu}$ is the 
Kostka-Foulkes polynomial, and $n(\mu)=\sum_{i\geq1}(i-1)\mu_i$, 
see~\cite{L2}. 

Let $V=\sk^N$, so that $\GL_N=\GL(V)$. For a pair of partitions
$(\lambda,\mu)$ such that $|\lambda|+|\mu|=N$ the corresponding unipotent
mirabolic character sheaf $\CF_{\lambda,\mu}$ on $\GL(V)\times V$
was constructed
in~\cite{FG}. On the other hand, the $GL_N$-orbits in the product of
the unipotent cone and $V$ are also numbered by the set of pairs of
partitions $(\lambda',\mu')$ such that $|\lambda'|+|\mu'|=N$
(see~\cite{T}). In Theorem~\ref{past} we compute the Frobenius trace
function of a mirabolic character sheaf $\CF_{\lambda,\mu}$ on an
orbit corresponding to $(\lambda',\mu')$. The answer is given in terms
of certain polynomials $\Pi_{(\lambda',\mu')(\lambda,\mu)}$, the
mirabolic analogues of the Kostka-Foulkes polynomials introduced
in~\cite{Sh}. More generally, 
in~\ref{medium} we compute the Frobenius trace functions (on any
orbit) of a wide class of Weil mirabolic character sheaves. These
trace functions form a basis in the space of $\GL_N(\Fq)$-invariant
functions on $\GL_N(\Fq)\times\Fq^N$, and we conjecture that the above
class of sheaves exhausts all the irreducible $\Gm$-equivariant Weil
mirabolic character sheaves. This would give a positive answer to a
question of G.~Lusztig. 

Recall that the Kostka-Foulkes polynomials are the matrix coefficients
of the transition matrix from the Hall-Littlewood basis to the Schur
basis of the ring $\Lambda$ of symmetric functions. Similarly, the
polynomials $\Pi_{(\lambda',\mu')(\lambda,\mu)}$ are the matrix
coefficients of the transition matrix from a certain {\em mirabolic
Hall-Littlewood} basis of $\Lambda\otimes\Lambda$ (introduced
in~\cite{Sh}) to the Schur basis, 
see~\ref{fast}. Recall that $\Lambda$ is isomorphic to the
Hall algebra~\cite{Mac} whose natural basis goes to the basis of
Hall-Littlewood polynomials. Similarly, $\Lambda\otimes\Lambda$ is
naturally isomorphic to a certain {\em mirabolic Hall bimodule} over
the Hall algebra, and then the natural basis of this bimodule goes to
the mirabolic Hall-Littlewood basis, see section~\ref{Hall}.
The structure constants of this basis, together with Green's formula
for the characters of $\GL_N(\Fq)$, enter the computation of the
Frobenius traces of the previous paragraph.

The Hall algebra is also closely related to the spherical Hecke
algebra
$\bH^\sph$ of $\GL_N$ (the convolution algebra of the affine Grassmannian
of $\GL_N$). Similarly, the mirabolic Hall bimodule is closely related
to a certain {\em spherical mirabolic bimodule} over $\bH^\sph$,
defined in terms of convolution of the affine Grassmannian and the
{\em mirabolic affine Grassmannian}, see section~\ref{Grass}.
The geometry of the mirabolic affine Grassmannian is a particular case
of the geometry of the {\em mirabolic affine flag variety} studied in
section~\ref{Flaff}. Both geometries are (mildly) semiinfinite.

Thus all the results of this note are consequences of a single 
guiding principle which may be loosely stated as follows: the
mirabolic substances form a bimodule over the classical ones; this
bimodule is usually free of rank one.

However, the affine mirabolic bimodule $\CR^\aff$ over the affine
Hecke algebra $\CH^\aff$ is not free (see Remark~\ref{hissing}).
Recall that $\CH^\aff$ can be realized in the
equivariant $K$-homology of the Steinberg variety. It would be very
interesting to find a similar realization of $\CR^\aff$. 

Finally let us mention that the results of this note are very closely related
to the results of~\cite{AH}, though our motivations are rather different.
The authors of~\cite{AH} were primarily interested in the geometry of 
enhanced nilpotent cone. They proved the parity vanishing of the IC stalks
of the orbit closures in the enhanced nilpotent cone, and identified the
generating functions of these stalks with Shoji's type-B Kostka polynomials.
Since the Schubert varieties in the mirabolic affine Grassmannian are
equisingular to the orbit closures in the enhanced nilpotent cone
(see~\ref{Lusztig},~\ref{Mir}), 
the appearance of Shoji's polynomials in the spherical mirabolic bimodule
and in the mirabolic Hall bimodule is an immediate corollary of~\cite{AH}.

\subsection{Acknowledgments}
We are indebted to G.~Lusztig for his
question about classification of mirabolic character sheaves over a
finite field, and also for bringing~\cite{Sol} to our attention.
We are obliged to P.~Achar and A.~Henderson for sending
us~\cite{AH} before its publication, and bringing~\cite{Sh} to our attention.
We are grateful to D.~Zakharov and R.~Yang for correction of the conditions
in~Propositions~\ref{aff thm_expl} and~\ref{aff thm_explicit} in the earlier version of our note.
It is a pleasure to thank V.~Lunts for his hospitality during our work
on this project. M.F.\ is also grateful to the Universit{\'e} Paris VI and IAS
for the hospitality and support; he was partially supported by the
Oswald Veblen Fund.
The work of V.G.\ was  partially supported by the NSF grant DMS-0601050.

\section{Mirabolic affine flags}
\label{Flaff}

\subsection{Notations}
We set $\bF=\sk((t)),\ \bO=\sk[[t]]$. Furthermore, $G=\GL(V)$, and
$\GF=G(\bF),\ \GO=G(\bO)$. The affine Grassmannian $\Gr=\GF/\GO$.
We fix a flag $F_\bullet\in\Fl(V)$, and its stabilizer Borel subgroup
$B\subset G$; it gives rise to an Iwahori subgroup $\bI\subset\GO$.
The affine flag variety $\Flaff=\GF/\bI$. We set $\bV=\bF\otimes_\sk V$, and
$\bVo=\bV-\{0\}$, and $\Pinf=\bVo/\sk^\times$. 

It is well known that the $\GF$-orbits in $\Flaff\times\Flaff$ are numbered
by the affine Weyl group $\gS_N^\aff$ formed by all the permutations $w$ 
of $\BZ$ such that $w(i+N)=w(i)+N$ for any $i\in\BZ$ ({\em periodic}
permutations). Namely, for a basis $\{e_1,\ldots,e_N\}$ of $V$ we set
$e_{i+Nj}:=t^{-j}e_i,\ i\in\{1,\ldots,N\},\ j\in\BZ$; then the following pair
$(F^1_\bullet,F^2_\bullet)$ of periodic flags of $\bO$-sublattices in $\bV$
lies in the orbit $\BO_w\subset\Flaff\times\Flaff$:

\begin{equation}
\label{first}
F^1_k=\langle e_k,e_{k-1},e_{k-2},\ldots\rangle,
F^2_k=\langle e_{w(k)},e_{w(k-1)},e_{w(k-2)},\ldots\rangle.
\end{equation}
(it is understood that $e_k,e_{k-1},e_{k-2},\ldots$ is a {\em topological}
basis of $F^1_k$).

Following~\cite{T},~Lemma~2, we define $RB^\aff$ as the set of pairs
$(w,\beta)$ where $w\in\gS_N^\aff$, and $\beta\subset\BZ$ such that if 
$i\in\BZ-\beta$, and $j\in\beta$, then either $i>j$ or $w(i)>w(j)$;
moreover, any $i\ll0$ lies in $\beta$, and any $j\gg0$ lies in $\BZ-\beta$.

\subsection{$\GF$-orbits in $\Flaff\times\Flaff\times\Pinf$}
The following proposition is an affine version of~\cite{Sol}, see also~\cite[2.11]{MWZ}.

\begin{prop}
\label{aff mwz}
There is a one-to-one correspondence between the set of $\GF$-orbits in
$\Flaff\times\Flaff\times\bVo$ (equivalently, in 
$\Flaff\times\Flaff\times\Pinf$) and $RB^\aff$.
\end{prop}

\begin{proof} The argument is entirely similar to the proof of~Lemma~2 
of~\cite{T}. It is left to the reader. We only mention that a representative
of an orbit corresponding to $(w,\beta)$ is given by 
$(F^1_\bullet,F^2_\bullet,v)$ where $(F^1_\bullet,F^2_\bullet)$ are as 
in~(\ref{first}), and $v=\sum_{k\in\beta}e_k$ (note that this infinite sum
makes sense in $\bV$). \end{proof}

\subsection{The mirabolic bimodule over the affine Hecke algebra}
Let $\sk=\Fq$, a finite field with $q$ elements. Then the affine Hecke
algebra of $G$ is the endomorphism algebra of the induced module
$H^\aff:=\End_\GF(\on{Ind}_\bI^\GF\BZ)$. It has the standard basis $\{T_w,\
w\in\gS_N^\aff\}$, and the structure constants are polynomial in $q$,
so we may and will view $H^\aff$ as the specialization under
$\bq\mapsto q$ of a $\BZ[\bq,\bq^{-1}]$-algebra $\bH^\aff$.
Clearly, $H=\End_\GF(\on{Ind}_\bI^\GF\BZ)$
coincides with the convolution ring of $\GF$-invariant functions on
$\Flaff\times\Flaff$. 

It acts by the right and left convolution on the
bimodule $R^\aff$ of $\GF$-invariant functions on
$\Flaff\times\Flaff\times\bVo$. For $\tw\in RB^\aff$ let
$T_\tw\in R^\aff$ stand for the characteristic function of the
corresponding orbit in $\Flaff\times\Flaff\times\bVo$. 
Note that the involutions 
$(F^1_\bullet,F^2_\bullet)\invl (F^2_\bullet,F^1_\bullet)$ and 
$(F^1_\bullet,F^2_\bullet,v)\invl (F^2_\bullet,F^1_\bullet,v)$ 
induce anti-automorphisms of the
algebra $\bH^\aff$ and the bimodule of $\GF$-invariant functions on
$\Flaff\times\Flaff\times\bVo$. These anti-automorphisms send
$T_w$ to $T_{w^{-1}}$ and $T_\tw$ to $T_{\tw^{-1}}$ where 
$\tw^{-1}=(w^{-1},w(\b))$ for $\tw=(w,\b)$.

We are going to describe the right action of $H^\aff$ on the
bimodule $R^\aff$ in the basis $\{T_\tw,\ \tw\in RB^\aff\}$ (and then the
formulas for the left action would follow via the above
anti-automorphisms). To this end recall that $H^\aff$ is generated
by $T_{s_1},\ldots,T_{s_N},T^{\pm1}_\tau$ where $T_{s_i}$ is the characteristic
function of the orbit formed by the pairs $(F^1_\bullet,F^2_\bullet)$
such that $F^1_j\ne F^2_j$ iff $j=i\pmod{N}$; and $\tau(k)=k+1,\ k\in\BZ$.
Evidently, $T_\tw T^{\pm1}_\tau=T_{\tw[\pm1]}$ where $\tw[\pm1]$ 
is the shift of $\tw$ by $\pm1$.
The following proposition is an affine version of~Proposition~2
of~\cite{T}, and the proof is straightforward.

\begin{prop}\label{aff thm_expl}
Let $\tw=(w,\b)\in RB^\aff$ and let $s=s_i\in\gS_N^\aff$, $i\in\{1,\dots,N\}$. Let $\s=\s(\tw)$ 
and $\s'=\s(\tw s)$ be given by the formula~(6) of~\cite{T}. 
Denote $\tw s = (ws, s(\b))$, and $\tw'=(w,\b\btu\{\sigma_i+1\})$, and $(\tilde{w}s)'=(ws,s(\beta)\setminus\{\sigma_{i+1}\})$,
where $\sigma_i$ is the maximal element in $\beta$ congruent to $i$ modulo $N$. Then
{\footnotesize\rm \eql{aff eq_expl}{  T_\tw T_s =
\begin{cases}
  T_{\tw s} &\text{if $ws > w$ and $\sigma_{i+1}\not\in\sigma'$,} \\
T_{\tw s} + T_{(\tw s)'} &
\text{if $ws > w$ and $\sigma_{i+1}\in\sigma'$,} \\
T_{\tw'} + T_{\tw's} &
\text{if $ws < w$ and $\sigma_i+1\not\in\beta$,} \\
(q-1) T_\tw + q T_{\tw s} &
\text{if $ws < w$ and $\sigma_i+1\in\beta\setminus\sigma$,} \\
(q-2) T_\tw + (q-1) (T_{\tw'} + T_{\tw s}) &
\text{if $ws < w$ and $\sigma_i+1\in\sigma$}. \\
\end{cases} }}
\end{prop}

\subsection{Modified bases}
The formulas~(\ref{aff eq_expl}) being polynomial in $q$, 
we may (and will) view the $H^\aff$-bimodule $R^\aff$ as the 
specialization under $\bq\mapsto q$ of the  
$\BZ[\bq,\bq^{-1}]$-bimodule $\bR^\aff$ over the 
$\BZ[\bq,\bq^{-1}]$-algebra $\bH^\aff$.
We consider a new variable $\bv,\ \bv^2=\bq$, and extend the scalars to 
$\BZ[\bv,\bv^{-1}]:\ 
\CH^\aff:=\BZ[\bv,\bv^{-1}]\otimes_{\BZ[\bq,\bq^{-1}]}\bH^\aff;\
\CR^\aff:=\BZ[\bv,\bv^{-1}]\otimes_{\BZ[\bq,\bq^{-1}]}\bR^\aff$.

Recall the basis $\{H_w:=(-\bv)^{-\ell(w)}T_w\}$ of $\CH^\aff$ 
(see e.g.~\cite{So}),
and the Kazhdan-Lusztig basis $\{\utH_w\}$ ({\em loc. cit.}); 
in particular, for $s_i\ (i=1,\ldots,N),\ \utH_{s_i}=H_{s_i}-\bv^{-1}$.
For $\tw=(w,\beta)\in RB^\aff$, we denote by $\ell(\tw)$ the sum
$\ell(w)+\ell(\beta)$ where $\ell(w)$ is the standard length function
on $\gS_N^\aff$, and $\ell(\beta)=\sharp(\beta\setminus\{-\BN\})-
\sharp(\{-\BN\}\setminus\beta)$. 
We introduce a new basis $\{H_\tw:=(-\bv)^{-\ell(\tw)}T_\tw\}$ of $\CR^\aff$.
In this basis the right action of the Hecke algebra generators $\utH_{s_i}$
takes the form:

\begin{prop}
  \label{aff thm_explicit}
Let $\tw=(w,\b)\in RB^\aff$ and let $s=s_i\in\gS_N^\aff$, $i\in\{1,\dots,N\}$. Let $\s=\s(\tw)$ 
and $\s'=\s(\tw s)$ be given by the formula~(6) of~\cite{T}. 
Denote $\tw s = (ws, s(\b))$, and $\tw'=(w,\b\btu\{\sigma_i+1\})$, and $(\tilde{w}s)'=(ws,s(\beta)\setminus\{\sigma_{i+1}\})$,
where $\sigma_i$ is the maximal element in $\beta$ congruent to $i$ modulo $N$. Then  

{\footnotesize\rm \eql{aff eq_explicit}{  H_\tw\utH_s =
\begin{cases}
  H_{\tw s}-\bv^{-1}H_\tw &\text{if $ws > w$ and $\sigma_{i+1}\not\in\sigma'$,} \\
H_{\tw s} - \bv^{-1}H_{(\tw s)'}-\bv^{-1}H_\tw &
\text{if $ws > w$ and $\sigma_{i+1}\in\sigma'$,} \\
H_{\tw'}-\bv^{-1}H_\tw-\bv^{-1}H_{\tw's} &
\text{if $ws < w$ and $\sigma_i+1\not\in\beta$,} \\
H_{\tw s}-\bv H_\tw &
\text{if $ws < w$ and $\sigma_i+1\in\beta\setminus\sigma$,} \\
(\bv^{-1}-\bv)H_\tw + (1-\bv^{-2})(H_{\tw'} + H_{\tw s}) &
\text{if $ws < w$ and $\sigma_i+1\in\sigma$}. \\
\end{cases} }}
\end{prop}

\subsection{Generators}
\label{kissing}
We consider the elements $\tw_{i,j}=(\tau^j,\beta_i)\in RB^\aff$ such that
$w=\tau^j$ (the shift by $j$), and
$\beta_i=\{i,i-1,i-2,\ldots\}$, for any $i,j\in\BZ$. 
The following lemma is proved exactly as~Corollary~2 of~\cite{T}.

\begin{lem}
\label{missing}
$\CR^\aff$ is generated by $\{\tw_{i,j},\ i,j\in\BZ\}$
as a $\CH^\aff$-bimodule.
\end{lem}

\begin{Rem} 
\label{hissing}
Let $\bP_\bF\subset\GF$ be the stabilizer of a vector $v\in\bVo$.
One can see easily that $\bR^\aff|_{\bq=q}$ is isomorphic to 
$\End_{\bP_\bF}(\on{Ind}_\bI^\GF\BZ)$ as a bimodule over
$\bH^\aff|_{\bq=q}=\End_\GF(\on{Ind}_\bI^\GF\BZ)$.
Let $Z^\aff\subset\CH^\aff$ stand for the center of $\CH^\aff$. 
Let $Z^\aff_{\on{loc}}$ stand for the field of fractions of $Z^\aff$.
Let $\CH^\aff_{\on{loc}}:=\CH^\aff\otimes_{Z^\aff}Z^\aff_{\on{loc}}$.
It is known that $\CH^\aff_{\on{loc}}\simeq\on{Mat}_{N!}({\mathbb Q})
\otimes_{\mathbb Q}Z^\aff_{\on{loc}}$. Let
$\CR^\aff_{\on{loc}}:=Z^\aff_{\on{loc}}\otimes_{Z^\aff}\CR^\aff
\otimes_{Z^\aff}Z^\aff_{\on{loc}}$. Then it follows from the main
theorem of~\cite{B} that $\CR^\aff_{\on{loc}}\simeq 
Z^\aff_{\on{loc}}\otimes_{\mathbb Q}\on{Mat}_{N!}({\mathbb Q})
\otimes_{\mathbb Q}Z^\aff_{\on{loc}}$. 
\end{Rem}

\subsection{Geometric interpretation}
\label{geom}
It is well known that $\CH^\aff$ is the Grothendieck ring (with respect to 
convolution) of the derived constructible $\bI$-equivariant category of
Tate Weil $\overline{\mathbb Q}_l$-sheaves on $\Flaff$, 
and multiplication by $\bv$
corresponds to the twist by $\overline{\mathbb Q}_l(-\frac{1}{2})$ 
(so that $\bv$ has 
weight 1). In particular, $H_w$ is the class of the
shriek extension of $\Ql[\ell(w)](\frac{\ell(w)}{2})$ 
from the corresponding orbit $\Flaff_w$, 
and $\utH_w$ is the selfdual class of the Goresky-MacPherson extension
of $\Ql[\ell(w)](\frac{\ell(w)}{2})$ from this 
orbit. We will interpret $\CR^\aff$ in a similar vein, as the
Grothendieck group of the derived constructible $\bI$-equivariant
category of Tate Weil $\overline{\mathbb Q}_l$-sheaves on $\Flaff\times\bVo$.

To be more precise, we view $\bV$ as an indscheme (of ind-infinite
type), the union of schemes (of infinite type)
$\bV_i:=t^{-i}\sk[[t]]\otimes V,\ i\in\BZ$. 
Here $\bV_i$ is the projective limit
of the finite dimensional affine spaces $\bV_i/\bV_j,\ j<i$.
Note that $\bI$ acts on $\bV_i$ linearly (over $\sk$), and it acts on 
any quotient $\bV_i/\bV_j$ through a finite dimensional quotient
group. Thus we have the derived constructible $\bI$-equivariant
category of Weil $\overline{\mathbb Q}_l$-sheaves on 
$\Flaff\times\bV_i/\bV_j$, 
to be denoted by $D_\bI(\Flaff\times\bV_i/\bV_j)$. For $j'<j$ we have the
inverse image functor from  $D_\bI(\Flaff\times\bV_i/\bV_j)$ to
$D_\bI(\Flaff\times\bV_i/\bV_{j'})$, and we denote by
$D_\bI(\Flaff\times\bVo_i)$ the 2-limit of this system. Now for $i'>i$ we
have the direct image functor from $D_\bI(\Flaff\times\bVo_i)$ to 
$D_\bI(\Flaff\times\bVo_{i'})$, and we denote by $D_\bI(\Flaff\times\bVo)$ the
2-limit of this system. 

Clearly, $D_\bI(\Flaff)$ acts by convolution both on the left and on
the right on $D_\bI(\Flaff\times\bVo)$.

The $\bI$-orbits in $\Flaff\times\bVo$ are numbered by $RB^\aff$; for
$\tw\in RB^\aff$, the locally closed embedding of the orbit
$\Omega_\tw\hookrightarrow\Flaff\times\bVo$ is denoted by $j^\tw$.

\begin{prop}
\label{tate}
For any $\tw\in RB^\aff$, the Goresky-MacPherson sheaf
$j^\tw_{!*}\Ql[\ell(\tw)](\frac{\ell(\tw)}{2})$ is Tate.
\end{prop}

\begin{proof} Repeats word for word the proof of Proposition~4
of~\cite{T}. For the base of induction, we use the fact that the orbit
closure $\bar{\Omega}_{\tw_{i,j}}$ (see~\ref{kissing}) is smooth. 
For the induction step we use the Demazure type resolutions as in 
{\em loc. cit.}
\end{proof}

\subsection{The completed bimodule $\hR$}
Let $D_\bI^\Tate(\Flaff)\subset D_\bI(\Flaff)$
(resp. $D_\bI^\Tate(\Flaff\times\bVo)\subset D_\bI(\Flaff\times\bVo)$) 
stand for the full subcategory of Tate sheaves. Then
$D_\bI^\Tate(\Flaff)$ is closed under convolution, and its $K$-ring is
isomorphic to $\CH^\aff$. The proof of Proposition~\ref{tate} implies
that $D_\bI^\Tate(\Flaff\times\bVo)$ is closed under both left and
right convolution with $D_\bI^\Tate(\Flaff)$. Hence
$K(D_\bI^\Tate(\Flaff\times\bVo))$ forms an $\CH^\aff$-bimodule. This
bimodule is isomorphic to a completion $\hR$ of $\CR^\aff$ we presently
describe. 

Recall that for an $\bO$-sublattice $F\subset\bV$ its virtual
dimension is $\dim(F):=\dim(F/(F\cap(\bO\otimes V)))-\dim((\bO\otimes
V)/(F\cap(\bO\otimes V)))$. Recall that $\bI$ is the stabilizer of the
flag $F^1_\bullet$, where $F^1_k=\langle
e_k,e_{k-1},e_{k-2},\ldots\rangle$. The connected components of
$\GF/\bI=\Flaff$ are numbered by $\BZ$: a flag $F_\bullet$ lies in the
component $\Flaff_i$ where $i=\dim(F_N)$. For the same reason, the
connected components of $\Flaff\times\bVo$ are numbered by $\BZ$: a
pair $(F_\bullet,v)$ lies in the connected component
$(\Flaff\times\bVo)_i$ where $i=\dim(F_N)$. We will say $\tw\in
RB^\aff_i$ iff $\Omega_\tw\subset(\Flaff\times\bVo)_i$.
Now note that for any $i,k\in\BZ$ there
are only finitely many $\tw\in RB^\aff$ such that
$\tw\in RB^\aff_i$ and $\ell(\tw)=k$.

We define $\hR$ as the direct sum $\hR=\bigoplus_{i\in\BZ}\hR_i$, and
$\hR_i$ is formed by all the formal sums $\sum_{\tw\in RB^\aff_i}a_\tw
H_\tw$ where $a_\tw\in\BZ[\bv,\bv^{-1}]$, and $a_\tw=0$ for
$\ell(\tw)\gg0$. 
So we have $K(D_\bI^\Tate(\Flaff\times\bVo))\simeq\hR$ as an
$\CH^\aff$-bimodule, and the isomorphism takes the class
$[j^\tw_!\Ql[\ell(\tw)](\frac{\ell(\tw)}{2})]$ to $H_\tw$.

\subsection{Bruhat order}
\label{bruhat}
Following Ehresmann and Magyar (see~\cite{M}) we will define a partial
order $\tw''\leq\tw'$ on a connected component $RB^\aff_i$.
Let $(F^1_\bullet,F'_\bullet,v')$
(resp. $F^1_\bullet,F''_\bullet,v''$) be a triple in the relative
position $\tw'$ (resp. $\tw''$). For any $k,j\in\BZ$ we define
$r_{jk}(\tw'):=\dim(F^1_j\cap F'_k)$. We also define
$\delta(j,k,\tw')$ to be 1 iff $v'\in(F^1_j+F'_k)$, and 0 iff
$v'\not\in(F^1_j+F'_k)$; we set $r_{\langle
  jk\rangle}(\tw'):=r_{jk}(\tw')+\delta(j,k,\tw')$. Finally, we define
$\tw''\leq\tw'$ iff $r_{jk}(\tw'')\geq r_{jk}(\tw')$, and  
$r_{\langle jk\rangle}(\tw'')\geq r_{\langle jk\rangle}(\tw')$ for all
$j,k\in\BZ$. 

The following proposition is proved similarly to the Rank Theorem~2.2
of~\cite{M}. 

\begin{prop}
\label{magyar}
For $\tw',\tw''\in RB^\aff_i$ the orbit $\Omega_{\tw''}$ lies in the
orbit closure $\bar{\Omega}_{\tw'}$ iff $\tw''\leq\tw'$.
\end{prop}

\subsection{Duality and the Kazhdan-Lusztig basis of $\hR$}
Recall that the Grothendieck-Verdier duality on $\Flaff$ induces
the involution (denoted by $h\mapsto\ol h$) of $\CH^\aff$ which takes
$\bv$ to $\bv^{-1}$ and $\utH_w$ to $\utH_w$. We will describe the
involution on $\hR$ induced by the Grothendieck-Verdier duality on
$\Flaff\times\bVo$. Recall the elements $\tw_{i,j}$ introduced
in~\ref{kissing}. We set $\utH_{\tw_{i,j}}:=\sum_{k\leq
  i}(-\bv)^{k-i}H_{\tw_{k,j}}$. This is the class
of the selfdual (geometrically constant) IC sheaf on the closure of the
orbit $\Omega_{\tw_{i,j}}$. The following proposition is proved
exactly as Proposition~5 of~\cite{T}.

\begin{prop}
\label{aff duality}
a) There exists a unique involution $r\mapsto\overline{r}$ on $\hR$ such
that $\overline\utH_{\tw_{i,j}}=\utH_{\tw_{i,j}}$ for any $i,j\in\BZ$, and
$\overline{hr}=\overline{h}\overline{r}$, and
$\overline{rh}=\overline{r}\overline{h}$ for any $h\in\CH^\aff$ and $r\in\hR$.

b) The involution in a) is induced by the Grothendieck-Verdier duality
on $\Flaff\times\bVo$. 
\end{prop}

The following proposition is proved exactly as Proposition~6
of~\cite{T}. 

\begin{prop}
\label{aff KL basis}
a) For each $\tw\in RB^\aff$ there exists a unique element $\utH_\tw\in\hR$
such that $\overline\utH_\tw=\utH_\tw$, and 
$\utH_\tw\in H_\tw+\sum_{\tilde{y}<\tw}\bv^{-1}\BZ[\bv^{-1}]H_{\tilde{y}}$.

b) For each $\tw\in RB^\aff$ the element $\utH_\tw$ is the class of the 
selfdual $\bI$-equivariant IC-sheaf with support $\bar{\Omega}_\tw$.
In particular, for $\tw=\tw_{i,j}$, the element $\utH_{\tw_{i,j}}$ 
is consistent with the notation introduced before 
Proposition~\ref{aff duality}.
\end{prop}

We conjecture that the sheaves 
$j_{!*}\Ql[\ell(\tw)](\frac{\ell(\tw)}{2})$
are pointwise pure. The parity vanishing of their stalks, and the 
positivity properties of the coefficients of the transition matrix
from $\{H_\tw\}$ to $\{\utH_\tw\}$ would follow.

\section{Mirabolic affine Grassmannian}
\label{Grass}

\subsection{$\GF$-orbits in $\Gr\times\Gr\times\Pinf$} 
We consider the spherical counterpart of the objects of
the previous section. To begin with, recall that the $\GF$-orbits in
$\Gr\times\Gr$ are numbered by the set $\gS_N^\sph$ formed by all the
nonincreasing $N$-tuples of integers
$\nu=(\nu_1\geq\nu_2\geq\ldots\geq\nu_N)$. Namely, for
such $\nu$, the following pair $(L^1,L^2)$ of $\bO$-sublattices in 
$\bV$ lies in the orbit ${\mathbb O}_\nu$:

\begin{equation}
\label{second}
L^1=\bO\langle e_1,e_2,\ldots,e_N\rangle,\ 
L^2=\bO\langle
t^{-\nu_1}e_1,t^{-\nu_2}e_2,\ldots,t^{-\nu_N}e_N\rangle. 
\end{equation}

We define $RB^\sph$ as $\gS_N^\sph\times\gS_N^\sph$. We have an
addition map $RB^\sph\to\gS_N^\sph:\ (\lambda,\mu)\mapsto\nu=\lambda+\mu$
where $\nu_i=\lambda_i+\mu_i,\ i=1,\ldots,N$.

\begin{prop}
\label{sph mwz}
There is a one-to-one correspondence between the set of $\GF$-orbits
in $\Gr\times\Gr\times\bVo$ (equivalently, in
$\Gr\times\Gr\times\Pinf$) and $RB^\sph$.
\end{prop}

\begin{proof} The argument is entirely similar to the proof of
Proposition~\ref{aff mwz}. We only mention that a
representative of an orbit $\BO_{(\lambda,\mu)}$
corresponding to $(\lambda,\mu)$ with $\lambda+\mu=\nu$ is given
by $(L^1,L^2,v)$ where $(L^1,L^2)$ are as in~(\ref{second}), and 
$v=\sum_{i=1}^Nt^{-\lambda_i}e_i$. \end{proof}

\begin{prop}
There is a one-to-one correspondence between the set of $\GF$-orbits
in $\Gr\times\Flaff\times\bVo$ (equivalently, in
$\Gr\times\Flaff\times\Pinf$) and the set of pairs of integer sequences 
$(\{b_1,\dots,b_N\},\{c_1,\dots,c_N\})$ such that if $b_i-i/N <b_j-j/N$
then $c_i\le c_j$. Namely, a representative of the orbit corresponding
to $(\{b_1,\dots,b_N\},\{c_1,\dots,c_N\})$ is given by $(L,F,v)$ where
$L=\bO\langle t^{b_1+c_1}e_1,\dots,t^{b_N+c_N}e_N\rangle$,
$F_k=\langle e_k,e_{k-1},e_{k-2},\dots\rangle$,
$v=\sum_{i=1}^N t^{b_i}e_i$. \qed
\end{prop}

\subsection{The spherical mirabolic bimodule}
\label{snova}
Let $\sk=\Fq$. Then the spherical affine Hecke $H^\sph$ 
algebra of $G$ is the
endomorphism algebra of the induced module
$\End_\GF(\on{Ind}_\GO^\GF\BZ)$. It coincides with the convolution
ring of $\GF$-invariant functions on $\Gr\times\Gr$.
It has the standard basis $\{U_\nu,\
\nu\in\gS^\sph_N\}$ of characteristic functions of $\GF$-orbits in
$\Gr\times\Gr$, and the structure constants are polynomial in $q$
(Hall polynomials), so we may and will view
$H^\sph=\End_\GF(\on{Ind}_\GO^\GF\BZ)$ as specialization of the
$\BZ[\bq,\bq^{-1}]$-algebra $\bH^\sph$ under $\bq\mapsto q$.

The algebra $H^\sph$ acts by the right and left convolution on the
bimodule $R^\sph$ of $\GF$-invariant functions on
$\Gr\times\Gr\times\bVo$. For $(\lambda,\mu)\in RB^\sph$ let
$U_{(\lambda,\mu)}$ stand for the characteristic function of the
corresponding orbit in $\Gr\times\Gr\times\bVo$. We are going to
describe the right and left action of $H^\sph$ on the bimodule in
the basis $\{U_{(\lambda,\mu)},\ (\lambda,\mu)\in RB^\sph\}$. To this
end recall that $H^\sph$ is a commutative algebra freely generated by
$U_{(1,0,\ldots,0)},U_{(1,1,0,\ldots,0)},\ldots,U_{(1,1,\ldots,1,0)}$,
and $U^{\pm1}$ where $U^{\pm1}$ is the characteristic function of the
orbit of $(L^1,t^{\mp1}L^1)$. We will denote
$\nu=(1,\ldots,1,0,\ldots,0)$ ($r$ 1's and $N-r$ 0's) by $(1^r)$.

Note that the assignment $\phi_{i,j}:\
(L_1,L_2,v)\mapsto(L_1,t^{-i-j}L_2,t^{-i}v)$ 
is a $\GF$-equivariant automorphism
of $\Gr\times\Gr\times\bVo$ sending an orbit $\BO_{(\lambda,\mu)}$ to
$\BO_{(\lambda+i^N,\mu+j^N)}$. We will denote the corresponding
automorphism of the bimodule $R^\sph$ by $\phi_{i,j}$ as well:
$\phi_{i,j}(U_{\lambda,\mu)}=U_{(\lambda+i^N,\mu+j^N)}$. 
Furthermore, an automorphism
$(L_1,L_2)\mapsto(L_2,L_1)$ of $\Gr\times\Gr$ induces an (anti)automorphism
$\varrho$ of (commutative) algebra $H^\sph,\ \varrho(U^{\pm1})=U^{\mp1},\
\varrho(U_\nu)=U_{\nu^*}$ where for $\nu=(\nu_1,\ldots,\nu_N)$ we set
$\nu^*=(-\nu_N,-\nu_{N-1},\ldots,-\nu_1)$. Similarly, an automorphism  
$(L_1,L_2,v)\mapsto(L_2,L_1,v)$ of $\Gr\times\Gr\times\bVo$ induces an
antiautomorphism $\varrho$ of the bimodule $R^\sph$ such that
$\varrho(U_{(\lambda,\mu)})=U_{(\mu^*,\lambda^*)}$, and
$\varrho(hm)=\varrho(m)\varrho(h)$ for any $h\in H^\sph,\ m\in R^\sph$.
Clearly, $U^{\pm1}U_{(\lambda,\mu)}=U_{(\lambda\pm1^N,\mu)}$, and
$U_{(\lambda,\mu)}U^{\pm1}=U_{(\lambda,\mu\pm1^N)}$.

\subsection{Structure constants}
\label{net}
In this subsection we will compute the structure constants 
$G^{(\lambda,\mu)}_{(1^r)(\lambda',\mu')}$ such that 
$U_{(1^r)}U_{(\lambda',\mu')}=\sum_{(\lambda,\mu)\in RB^\sph} 
 G^{(\lambda,\mu)}_{(1^r)(\lambda',\mu')}U_{(\lambda,\mu)}$
(see Proposition~\ref{az} below). 
Due to the existence of the automorphisms $\phi_{i,j}$ of $R^\sph$, it
suffices to compute $G^{(\lambda,\mu)}_{(1^r)(\lambda',\mu')}$ for
$\lambda',\mu'\in\BN^N$. In this case $\lambda,\mu$ necessarily lie in
$\BN^N$ as well, that is, all the four $\lambda',\mu',\lambda,\mu$ are
partitions (with $N$ parts). We have $\lambda=(\lambda_1,\ldots,\lambda_N)$;
we may and will assume that $\lambda_1>0$.
We set $n:=|\lambda|+|\mu|$, and let
$D=\sk^n$. We fix a nilpotent endomorphism $u$ of $D$, and a vector
$v\in D$ such that the type of $\GL(D)$-orbit of the pair $(u,v)$ is
$(\lambda,\mu)$ (see~\cite{T}, Theorem~1). By the definition of the
structure constants in the spherical mirabolic bimodule, 
$G^{(\lambda,\mu)}_{(1^r)(\lambda',\mu')}$ is the number of
$r$-dimensional vector subspaces $W\subset\on{Ker}(u)$ such that the
type of the pair $(u|_{D/W}, v\pmod{W})$ is $(\lambda',\mu')$. 

To formulate the answer we need to introduce certain auxiliary data
in $\on{Ker}(u)$. First of all, $u^{\lambda_1-1}v$ is a nonzero vector
in $\on{Ker}(u)$. We consider the pair of partitions
$(\nu,\theta)=\Upsilon(\lambda,\mu)$ (notations introduced before
Corollary~1 of~\cite{T}), so that
$\nu=\lambda+\mu$ is the Jordan type of $u$. We consider the dual
partitions $\tilde{\nu},\tilde{\theta}$. We consider the following
flag of subspaces of $\on{Ker}(u)$:
\begin{multline}
F^{\tilde{\nu}_{\nu_1}}:=\on{Ker}(u)\cap\on{Im}(u^{\nu_1-1})\subset
F^{\tilde{\nu}_{\nu_2}}:=\on{Ker}(u)\cap\on{Im}(u^{\nu_2-1})\subset\ldots\\
\subset
F^{\tilde{\nu}_2}:=\on{Ker}(u)\cap\on{Im}(u^{\nu_{\tilde{\nu}_2}-1})
\subset F^{\tilde{\nu}_1}:=\on{Ker}(u).\nonumber
\end{multline}
It is (an incomplete, in general) flag of intersections of
$\on{Ker}(u)$ with the images of $u,u^2,u^3,\ldots$. More precisely, 
for any $k=0,1,\ldots,\nu_1$ we have 
$F_k:=\on{Ker}(u)\cap\on{Im}(u^k)=F^{\tilde{\nu}_{k+1}}$, and
$\dim(F^{\tilde{\nu}_{k+1}})=\tilde{\nu}_{k+1}$.
There is a unique $k_0$ such that $u^{\lambda_1-1}v\in F_{k_0}$ but
$u^{\lambda_1-1}v\not\in F_{k_0+1}$; namely, we choose the maximal $i$
such that $\lambda_i=\lambda_1$, and then $k_0=\nu_i-1$. 

Let $Q\subset\GL(\on{Ker}(u))$ be the stabilizer of the flag
$F_\bullet$, a parabolic subgroup of $\GL(\on{Ker}(u))$;  
and let $Q'\subset Q$
be the stabilizer of the vector $u^{\lambda_1-1}v$. Both $Q$
and $Q'$ have finitely many orbits in the Grassmannian $\on{Gr}$ of
$r$-dimensional subspaces in $\on{Ker}(u)$. The orbits of $Q$ are
numbered by the compositions
$\rho=(\rho_1,\ldots,\rho_{\nu_1})$ such that $|\rho|=r$, and
$0\leq\rho_k\leq\tilde{\nu}_k-\tilde{\nu}_{k+1}$. Namely,
$W\in\on{Gr}$ lies in the orbit $\BO_\rho$ iff $\dim(W\cap
F_k)=\rho_{k+1}+\ldots+\rho_{\nu_1}$; equivalently,
$\dim(W+F_k)=\tilde{\nu}_{k+1}+\rho_1+\ldots+\rho_k$. 
If we extend the flag $F_\bullet$ to a complete flag in $\on{Ker}(u)$,
then the stabilizer of the extended flag is a Borel subgroup $B\subset Q$.
The orbit $\BO_\rho$ is a union of certain $B$-orbits
in $\on{Gr}$, that is Schubert cells. So the cardinality of $\BO_\rho$
is a sum of powers of $q$ given by the well known formula for the
dimension of the Schubert cells (see e.g. Appendix to Chapter~II
of~\cite{Mac}). We will denote this cardinality by $P_\rho$. 
Note that the Jordan type of $u|_{D/W}$ for $W\in\BO_\rho$ is
$\nu':=\rho(\nu)$ where $\rho(\nu)$ is defined as the partition dual
to $\tilde{\nu}'=(\tilde{\nu}'_1,\tilde{\nu}'_2,\ldots)$, and
$\tilde{\nu}'_k:=\tilde{\nu}_{k+1}+\dim(W+F_{k-1})-\dim(W+F_k)=
\tilde{\nu}_k-\rho_k$.
  
Now each $Q$-orbit $\BO_\rho$ in $\on{Gr}$ splits as a union
$\BO_\rho=\bigsqcup_{0\leq j\leq\nu_1}\BO_{\rho,j}$ of
$Q'$-orbits. Namely, $W\in\BO_\rho$ lies in $\BO_{\rho,j}$ iff
$u^{\lambda_1-1}v\in W+F_j$ but $u^{\lambda_1-1}v\not\in W+F_{j+1}$
(so that for some $j$, e.g. $j<k_0,\ \BO_{\rho,j}$ may be empty). 
The type of $(u|_{D/W},v\pmod{W})$ for $W\in\BO_{\rho,j}$ is
$(\nu',\theta'):=(\rho,j)(\nu,\theta)$ where $\nu'=\rho(\nu)$, and 
$\theta'$ is defined as the partition dual to $\tilde{\theta}'=
(\tilde{\theta}'_1,\tilde{\theta}'_2,\ldots)$, and 
$\tilde{\theta}'_k:=\tilde{\theta}_{k+1}+
\dim(W+F_{k-1}+\sk u^{\lambda_1-1}v)-\dim(W+F_k+\sk
u^{\lambda_1-1}v)$.
Finally, note that $\dim(W+F_{k-1}+\sk u^{\lambda_1-1}v)-\dim(W+F_k+\sk
u^{\lambda_1-1}v)=\dim(W+F_{k-1})-\dim(W+F_k)=
\tilde{\nu}_k-\tilde{\nu}_{k+1}-\rho_k$ if $j\ne k-1$, and
$\dim(W+F_{k-1}+\sk u^{\lambda_1-1}v)-\dim(W+F_k+\sk
u^{\lambda_1-1}v)=\dim(W+F_{k-1})-\dim(W+F_k)-1=
\tilde{\nu}_k-\tilde{\nu}_{k+1}-\rho_k-1$ if $j=k-1$.

It remains to find the cardinality $P_{\rho,j}$ of $\BO_{\rho,j}$.
Let us denote $u^{\lambda_1-1}v$ by $v'$ for short. Then $v'\in
F_{k_0},\ v'\in W+F_j,\ v'\not\in F_{k_0+1},\ v'\not\in W+F_{j+1}$,
thus $v'\in A:=\left\{(W+F_j)\cap F_{k_0}\right\}\setminus
\left(\left\{(W+F_j)\cap 
F_{k_0+1}\right\}\cup\left\{(W+F_{j+1})\cap F_{k_0}\right\}\right)$.
The cardinality of $A$ equals $P_A:=q^{\dim(W+F_j)\cap F_{k_0}}-
q^{\dim(W+F_j)\cap F_{k_0+1}}-q^{\dim(W+F_{j+1})\cap F_{k_0}}+
q^{\dim(W+F_{j+1})\cap F_{k_0+1}}$, while for any $i>l$ we have 
$\dim(W+F_i)\cap F_l=\dim(W+F_i)+\dim F_l-\dim(W+F_l)=
\tilde{\nu}_{i+1}+\rho_{l+1}+\ldots+\rho_i$. Now we can count the set
of pairs $(W,v')$ in a relative position $(\rho,j)$ with respect 
to $F_\bullet$ in two ways. First all $v'$ in $F_{k_0}\setminus
F_{k_0+1}$ ($q^{\tilde{\nu}_{k_0+1}}-q^{\tilde{\nu}_{k_0+2}}$ choices
altogether), and then for each $v'$ all $W$ in $\BO_{\rho,j}$
($P_{\rho,j}$ choices altogether). Second, all $W$ in $\BO_\rho$
($P_\rho$ choices altogether), and then for each $W$ all $v'$ in $A$
($P_A$ choices altogether). We find 
\begin{equation}
\label{phobos}
P_{\rho,j}=P_\rho\cdot
P_A/(q^{\tilde{\nu}_{k_0+1}}-q^{\tilde{\nu}_{k_0+2}})
\end{equation}
Note that $P_{\rho,j}$ is a polynomial in $q$. We conclude that this
polynomial computes the desired structure constant 
\begin{equation}
\label{demos}
G^{(\lambda,\mu)}_{(1^r)(\lambda',\mu')}=P_{\rho,j}
\end{equation}
where $(\lambda',\mu')=\Xi(\nu',\theta')$ (notations introduced before
Corollary~1 of~\cite{T}), and $(\nu',\theta')=(\rho,j)(\nu,\theta)$
where as before we have $(\nu,\theta)=\Upsilon(\lambda,\mu)$. 
   
Clearly, for any $i,j\geq0$ we have
$G^{(\lambda,\mu)}_{(1^r)(\lambda',\mu')}= 
G^{(\lambda+i^N,\mu+j^N)}_{(1^r)(\lambda'+i^N,\mu'+j^N)}$.
Hence for any $(\lambda,\mu),(\lambda',\mu')\in RB^\sph$ we can set
$G^{(\lambda,\mu)}_{(1^r)(\lambda',\mu')}:= 
G^{(\lambda+i^N,\mu+j^N)}_{(1^r)(\lambda'+i^N,\mu'+j^N)}$ for any
$i,j\gg0$.

Also, we set
\begin{equation}
\label{zhut'}
G^{(\lambda,\mu)}_{(\lambda',\mu')(1^r)}:=
G^{(\mu^*-1^N,\lambda^*)}_{(1^{N-r})(\mu'{}^*,\lambda'{}^*)}.
\end{equation}

Thus we have proved the following proposition (the second statement is
equivalent to the first one via the antiautomorphism $\varrho$).

\begin{prop}
\label{az}
Let $(\lambda',\mu')\in RB^\sph$, and $1\leq r\leq N-1$. Then 
\begin{multline}
\label{AZ}
U_{(1^r)}U_{(\lambda',\mu')}=\sum_{(\lambda,\mu)\in RB^\sph} 
 G^{(\lambda,\mu)}_{(1^r)(\lambda',\mu')}U_{(\lambda,\mu)},\ and\\
U_{(\lambda',\mu')}U_{(1^r)}=\sum_{(\lambda,\mu)\in RB^\sph}
G^{(\lambda,\mu)}_{(\lambda',\mu')(1^r)}U_{(\lambda,\mu)}.
\end{multline} 
\end{prop}

\subsection{Modified bases and generators}
\label{sph gen}
The formulas~(\ref{AZ}) being polynomial in $q$, we may and will view
the $H^\sph$-bimodule $R^\sph$ of $\GF$-invariant functions on
$\Gr\times\Gr\times\bVo$ as the specialization under $\bq\mapsto q$ of
a $\BZ[\bq,\bq^{-1}]$-bimodule $\bR^\sph$ over the
$\BZ[\bq,\bq^{-1}]$-algebra $\bH^\sph$. 
We extend the scalars to 
$\BZ[\bv,\bv^{-1}]:\ 
\CH^\sph:=\BZ[\bv,\bv^{-1}]\otimes_{\BZ[\bq,\bq^{-1}]}\bH^\sph;\
\CR^\sph:=\BZ[\bv,\bv^{-1}]\otimes_{\BZ[\bq,\bq^{-1}]}\bR^\sph$.

Recall the selfdual basis $C_\lambda$ of $\CH^\sph$
(see. e.g.~\cite{L2}). In particular, for $1\leq r\leq N-1,\
C_{(1^r)}=(-\bv)^{-r(N-r)}U_{(1^r)}$. For $(\lambda,\mu)\in RB^\sph$
with $\nu=\lambda+\mu$, we denote by $\ell(\lambda,\mu)$ the sum
$d(\nu)+|\lambda|$ with $|\lambda|:=\lambda_1+\ldots+\lambda_N$, and
$d(\nu):=|\nu|(N-1)-2n(\nu)$ where $n(\nu)=\sum_{i=1}^N(i-1)\nu_i$. 

We introduce a new basis
$\{H_{(\lambda,\mu)}:=(-\bv)^{-\ell(\lambda,\mu)}U_{(\lambda,\mu)}$ of
$\CR^\sph$. We consider the elements
$(i^N,j^N)=((i,\ldots,i),(j,\ldots,j))\in RB^\sph$ for any
$i,j\in\BZ$. The following lemma is proved the same way
as~Lemma~\ref{missing}. 
 
\begin{lem}
\label{sph missing}
$\CR^\sph$ is generated by $\{H_{(i^N,j^N)},\ i,j\in\BZ\}$ as an
$\CH^\sph$-bimodule. 
\end{lem}

\subsection{Geometric interpretation and the completed bimodule
  $\hRs$} 
Following the pattern of subsection~\ref{geom} we define the category
$D_\GO(\Gr\times\bVo)$ acted by convolution (both on the left and on
the right) by $D_\GO(\Gr)$. Similarly to Proposition~\ref{tate}, we
have (in obvious notations):

\begin{prop}
\label{sph tate}
For any $(\lambda,\mu)\in RB^\sph$, the Goresky-MacPherson sheaf
$j^{(\lambda,\mu)}_{!*}\Ql[\ell(\lambda,\mu)](\frac{\ell(\lambda,\mu)}{2})$ 
is Tate.
\end{prop}

We also have the full subcategories of Tate sheaves
$D^\Tate_\GO(\Gr)\subset D_\GO(\Gr)$ and
$D^\Tate_\GO(\Gr\times\bVo)\subset D_\GO(\Gr\times\bVo)$. Furthermore,  
$D^\Tate_\GO(\Gr)$ is closed under convolution, and 
$D^\Tate_\GO(\Gr\times\bVo)$ is closed under both right and left
convolution with $D^\Tate_\GO(\Gr)$. The $K$-ring
$K(D^\Tate_\GO(\Gr))$ is isomorphic to $\CH^\sph$, and this
isomorphism sends the class of the selfdual Goresky-MacPherson sheaf
on the orbit closure $\Gr_\lambda$ to $C_\lambda$. The $K$-group
$K(D^\Tate_\GO(\Gr\times\bVo))$ forms an $\CH^\sph$-bimodule
isomorphic to a completion $\hRs$ of $\CR^\sph$ we presently
describe. 

The connected components of $\Gr\times\bVo$ are numbered by $\BZ$: a
pair $(L,v)$ lies in the connected component $(\Gr\times\bVo)_i$ where
$i=\dim(L)$. We will say that $(\lambda,\mu)\in RB^\sph_i$ if the
corresponding orbit lies in $(\Gr\times\bVo)_i$; equivalently,
$\sum_{j=1}^N\lambda_j+\sum_{j=1}^N\mu_j=i$. Note that for any
$i,k\in\BZ$ there are only finitely many $(\lambda,\mu)\in RB^\sph$
such that $(\lambda,\mu)\in RB^\sph_i$, and $\ell(\lambda,\mu)=k$. 

We define $\hRs$ as the direct sum $\hRs=\bigoplus_{i\in\BZ}\hRs_i$, and
$\hRs_i$ is formed by all the formal sums $\sum_{(\lambda,\mu)\in
  RB^\sph_i}a_{(\lambda,\mu)} 
H_{(\lambda,\mu)}$ where $a_{(\lambda,\mu)}\in\BZ[\bv,\bv^{-1}]$, and 
$a_{(\lambda,\mu)}=0$ for $\ell(\lambda,\mu)\gg0$. 
So we have $K(D_\GO^\Tate(\Gr\times\bVo))\simeq\hRs$ as an
$\CH^\sph$-bimodule, and the isomorphism takes the class
$[j^{(\lambda,\mu)}_!\Ql[\ell(\lambda,\mu)](\frac{\ell(\lambda,\mu)}{2})]$ 
to $H_{(\lambda,\mu)}$.

\subsection{Bruhat order, duality and the Kazhdan-Lusztig basis}
Following Achar and Henderson~\cite{AH}, we define a partial order
$(\lambda,\mu)\leq(\lambda',\mu')$ on a connected component
$RB^\sph_i$: we say $(\lambda,\mu)\leq(\lambda',\mu')$ iff
$\lambda_1\leq\lambda'_1,\ \lambda_1+\mu_1\leq\lambda'_1+\mu'_1,\ 
\lambda_1+\mu_1+\lambda_2\leq\lambda'_1+\mu'_1+\lambda'_2,\
\lambda_1+\mu_1+\lambda_2+\mu_2\leq\lambda'_1+\mu'_1+\lambda'_2+\mu'_2,
\ldots$ (in the end we have $\sum_{k=1}^N\lambda_k+\sum_{k=1}^N\mu_k=
\sum_{k=1}^N\lambda'_k+\sum_{k=1}^N\mu'_k=i$).
The following proposition is due to Achar and Henderson (Theorem~3.9
of~\cite{AH}) :

\begin{prop}
\label{sph bruhat}
For $(\lambda,\mu),(\lambda',\mu')\in RB^\sph_i$ the $\GO$-orbit
$\Omega_{(\lambda,\mu)}\subset\Gr\times\bVo$ lies in the orbit closure 
$\bar{\Omega}_{(\lambda',\mu')}$ iff
$(\lambda,\mu)\leq(\lambda',\mu')$. 
\end{prop}

Now we will describe the involution on $\hRs$ induced by the
Grothen\-dieck-Verdier duality on $\Gr\times\bVo$. Recall the elements
$(i^N,j^N)$ introduced in~\ref{sph gen}. We set
$\utH_{(i^N,j^N)}:=\sum_{k\leq0}(-\bv)^{Nk}H_{((i-k)^N,(j+k)^N)}$.
This is the class of the selfdual (geometrically constant) IC sheaf on
the closure of the orbit $\Omega_{(i^N,j^N)}$. The following
propositions are proved exactly as Propositions~\ref{aff duality}
and~\ref{aff KL basis}: 

\begin{prop}
\label{sph duality}
a) There exists a unique involution $r\mapsto\overline{r}$ on $\hRs$ such
that $\overline\utH_{(i^N,j^N)}=\utH_{(i^N,j^N)}$ for any $i,j\in\BZ$, and
$\overline{hr}=\overline{h}\overline{r}$, and
$\overline{rh}=\overline{r}\overline{h}$ for any $h\in\CH^\sph$ and $r\in\hRs$.

b) The involution in a) is induced by the Grothendieck-Verdier duality
on $\Gr\times\bVo$. 
\end{prop}

\begin{prop}
\label{sph KL basis}
a) For each $(\lambda,\mu)\in RB^\sph$ there exists a unique element 
$\utH_{(\lambda,\mu)}\in\hRs$
such that $\overline\utH_{(\lambda,\mu)}=\utH_{(\lambda,\mu)}$, and 
$\utH_{(\lambda,\mu)}\in H_{(\lambda,\mu)}+
\sum_{(\lambda',\mu')<(\lambda,\mu)}\bv^{-1}\BZ[\bv^{-1}]H_{(\lambda',\mu')}$.

b) For each $(\lambda,\mu)\in RB^\sph$ the element 
$\utH_{(\lambda,\mu)}$ is the class of the 
selfdual $\GO$-equivariant IC-sheaf with support 
$\bar{\Omega}_{(\lambda,\mu)}$.
In particular, for $(\lambda,\mu)=(i^N,j^N)$, the element $\utH_{(i^N,j^N)}$ 
is consistent with the notation introduced before 
Proposition~\ref{sph duality}.
\end{prop}

We will write
\begin{equation}
\label{ginz}
\utH_{(\lambda,\mu)}=\sum_{(\lambda',\mu')\leq(\lambda,\mu)} 
\Pi_{(\lambda',\mu'),(\lambda,\mu)}H_{(\lambda',\mu')}.
\end{equation}
The coefficients $\Pi_{(\lambda',\mu'),(\lambda,\mu)}$ are polynomials
in $\bv^{-1}$. As we will see in subsection~\ref{fast} below, 
they coincide with
a generalization of Kostka-Foulkes polynomials introduced by Shoji 
in~\cite{Sh}. 

We define a sub-bimodule $\tR\subset\hRs$ generated ({\em not}
topologically) by the set $\utH_{(\lambda,\mu)},\ (\lambda,\mu)\in
RB^\sph$. It turns out to be a free $\CH^\sph$-bimodule of rank one:

\begin{thm}
\label{summer}
$C_\lambda\utH_{(0^N,0^N)}C_\mu=\utH_{(\lambda,\mu)}$.
\end{thm}

The proof will be given in subsection~\ref{winter} after we introduce the
necessary ingredients in~\ref{Lusztig} and~\ref{Mir}.

\subsection{Lusztig's construction}
\label{Lusztig}
Following Lusztig (see~\cite{L2}, section~2) 
we will prove that the $G$-orbit closures
in $\N\times V$ are equisingular to (certain open pieces of) the
$\GO$-orbit closures in $\Gr\times\bVo$. So we set
$E=V\oplus\ldots\oplus V$ ($N$ copies), and let $t:\ E\to E$ be
defined by $t(v_1,\ldots,v_N)=(0,v_1,\ldots,v_{N-1})$. Let $\CY$ be the
variety of all pairs $(E',e)$ where $E'\subset E$ is an
$N$-dimensional $t$-stable subspace, and $e\in E'$. Let $\CY_0$ be the
open subvariety of $\CY$ consisting of those pairs $(E',e)$ in which
$E'$ is transversal to $V\oplus\ldots\oplus V\oplus0$. According
to {\em loc. cit.} $\CY_0$ is isomorphic to $\CN\times V$, the
isomorphism sending $(u,v)$ to
$\left(E'=(u^{N-1}w,u^{N-2}w,\ldots,uw,w)_{w\in
    V},e=(u^{N-1}v,u^{N-2}v,\ldots,uv,v)\right)$. 
Now $E$ is naturally
isomorphic to $\left(t^{-N}\sk[[t]]/\sk[[t]]\right)\otimes V$ 
(together with the
action of $t$), and the assignment $(E',e)\mapsto(L:=E'\oplus
\sk[[t]]\otimes V,e)$ embeds $\CY$ into
$\Gr_{(N,0,\ldots,0)}\times\bV$. We will denote the composed embedding
$\CN\times V\hookrightarrow\Gr\times\bV$ by
$\psi:\ (u,v)\mapsto(L(u,v),e(u,v))$. There is an open subset 
$\CW\subset\sk[[t]]\otimes V$ with the property that for any
$w\in\CW$, and any $(u,v)\in(\CN\times V)_{(\lambda,\mu)}$ (a
$G$-orbit, see~\cite{T},~Theorem~1), we have
$(L(u,v),e(u,v)+w)\in\Omega_{(\lambda,\mu)}$ (the corresponding
$\GO$-orbit in $\Gr\times\bVo$). Moreover, the resulting embedding
$\CW\times(\CN\times V)_{(\lambda,\mu)}\hookrightarrow
\Omega_{(\lambda,\mu)}$ is an open embedding. Also, the embedding
$\CW\times\overline{(\CN\times V)}_{(\lambda,\mu)}\hookrightarrow
\bar{\Omega}_{(\lambda,\mu)}$ of the orbit closures is an open
embedding as well. Hence the Frobenius action on the IC stalks of
$\overline{(\CN\times V)}_{(\lambda,\mu)}$ is encoded in the polynomials
$\Pi_{(\lambda',\mu'),(\lambda,\mu)}$ 
introduced after Proposition~\ref{sph KL basis}.

\subsection{Mirkovi\'c-Vybornov construction}
\label{Mir}
The $\GO$-orbits $\Omega_{(\lambda,\mu)}\subset\Gr\times\bVo$
considered in subsection~\ref{Lusztig} are rather special: all the
components $\lambda_k,\mu_k$ are nonnegative integers, and
$\sum_{k=1}^N\lambda_k+\sum_{k=1}^N\mu_k=N$. To relate the
singularities of more general orbit closures
$\bar{\Omega}_{(\lambda',\mu')}$ to the singularities of orbits in the
enhanced nilpotent cones (for different groups $\GL_n,\ n\ne N$) we
need a certain generalization of Lusztig's construction, due to
Mirkovi\'c and Vybornov~\cite{MV}. 

To begin with, note that the assignment $\phi_{i,j}:\
(L,v)\mapsto(t^{-i-j}L,t^{-i}v)$ is a $\GO$-equivariant automorphism
of $\Gr\times\bVo$ sending $\Omega_{(\lambda,\mu)}$ to
$\Omega_{(\lambda+i^N,\mu+j^N)}$. Thus we may restrict ourselves to
the study of orbits $\Omega_{(\lambda,\mu)}$ with
$\lambda,\mu\in\BN^N$ without restricting generality. Geometrically,
this means to study the pairs $(L,v)$ such that $L\supset
L^1=\bO\langle e_1,\ldots,e_N\rangle$ and $L\ni v\not\in L^1$. 

Let $n=rN$ for $r\in\BN$. We consider an $n$-dimensional $\sk$-vector
space $D$ with a basis $\{e_{k,i},\ 1\leq k\leq r,\ 1\leq i\leq N\}$
and a nilpotent endomorphism $x:\ e_{k,i}\mapsto e_{k-1,i},\
e_{1,i}\mapsto0$. The Mirkovi\'c-Vybornov transversal slice is defined
as $T_x:=\{x+f,\ f\in\End(D):\ f^{l,j}_{k,i}=0$ if $k\ne r\}$.
Its intersection with the nilpotent cone of $\End(D)$ is
$T_x\cap\CN_n$.
 
Let $L^2\in\Gr$ be given as $L^2=t^{-r}L^1$. It lies in the orbit
closure $\Gr_{(n,0,\ldots,0)}$, and we will describe an open
neighbourhood $\CU$ of $L^2$ in $\Gr_{(n,0,\ldots,0)}$ isomorphic to
$T_x\cap\CN_n$. 
We choose a direct complement to $L^2$ in $\bV$ so that 
$L_2:=t^{-r-1}\sk[t^{-1}]\otimes V$. Then $\CU$ is formed by all the
lattices whose projection along $L_2$ is an isomorphism onto $L^2$.
Any $L\in\CU$ is of the form $(1+g)L^2$ where $g:\ L^2\to L_2$ is a
linear map with the kernel containing $L^1$, i.e. $g:\ L^2/L^1\to
L_2$. Now we use the natural identification of
$L^2/L^1$ with $D$ (so that the action of $t$ corresponds to the
action of $x$). Furthermore, we identify $t^{-r}V$ with a subset of
$L^2/L^1=D$. Hence we may view $g$ as a sum $\sum_{k=1}^\infty
t^{-k}g_k$ where $g_k:\ D\to t^{-r}V$ are linear maps.
Composing with $t^{-r}V\hookrightarrow D$ we may view $g_k$ as an
endomorphism of $D$. Then $L$ being a lattice is equivalent to the
condition: $g_k=g_1(t+g_1)^{k-1}$ and $t+g_1$ is nilpotent.
In other words, the desired isomorphism $T_x\cap\CN_n\iso\CU$ is of
the form:
$$T_x\cap\CN_n\ni x+f\mapsto L=L(x+f):=
\left(1+\sum_{k=1}^\infty t^{-k}f(t+f)^{k-1}\right)L^2.$$

Now we identify $D$ with $t^{-1}V\oplus\ldots\oplus t^{-r}V\subset L^2$.
Given a vector $v\in D$ we consider its image $v\in L^2$ under the
above embedding, and define $e(x+f,v)\in L(x+f)$ as the preimage of $v$
under the isomorphism $L\iso L^2$ (projection along $L_2$).
Thus we have constructed an embedding $\psi:\ (T_x\cap\CN_n)\times
D\hookrightarrow\Gr\times\bV,\ (x+f,v)\mapsto L(x+f),e(x+f,v)$.
Note that the Jordan type of any nilpotent $x+f$ is given by a
partition $\nu$ with the number of parts less than or equal to $N$.
There is an open subset 
$\CW\subset\sk[[t]]\otimes V$ with the property that for any
$w\in\CW$, and any $(x+f,v)\in((T_x\cap\CN_n)\times V)_{(\lambda,\mu)}$ (the
intersection with a $\GL_n$-orbit), we have
$(L(x+f),e(x+f,v)+w)\in\Omega_{(\lambda,\mu)}$ (the corresponding
$\GO$-orbit in $\Gr\times\bVo$). Moreover, the resulting embedding
$\CW\times((T_x\cap\CN_n)\times D)_{(\lambda,\mu)}\hookrightarrow
\Omega_{(\lambda,\mu)}$ is an open embedding. Also, the embedding
$\CW\times\overline{((T_x\cap\CN_n)\times V)}_{(\lambda,\mu)}\hookrightarrow
\bar{\Omega}_{(\lambda,\mu)}$ of the intersection with the orbit 
closure is an open embedding as well.

We conclude that the orbit closures $\bar{\Omega}_{(\lambda,\mu)}$
with $\sum_{k=1}^N\lambda_k+\sum_{k=1}^N\mu_k$ divisible by $N$ are
equisingular to certain $\GL_n$-orbit closures in $\CN_n\times D$ for
some $n$ divisible by $N$.

\subsection{Semismallness of convolution}
\label{winter}
We are ready for the proof of Theorem~\ref{summer}.
Let us denote the self-dual Goresky-MacPherson sheaf on the
orbit $\Gr_\lambda$ (whose class is $C_\lambda$) by  
$IC_\lambda$ for short. Then the convolution power
$IC_{(1,0,\ldots,0)}^{*l}$ is isomorphic to
$\oplus_{|\lambda|=l}K_\lambda\otimes IC_\lambda$ for certain vector
spaces $K_\lambda$ (equal to the multiplicities of irreducible
$\GL_N$-modules in $V^{\otimes l}$). We stress that $K_\lambda$ is
concentrated in degree 0, that is convolution morphism is stratified
semismall. 
Thus it suffices to prove

\begin{equation}
\label{obr}
IC_{(1,0,\ldots,0)}^{*l}*IC_{(0^N,0^N)}*IC_{(1,0,\ldots,0)}^{*m}\simeq 
\bigoplus_{|\mu|=m}^{|\lambda|=l}K_\mu\otimes K_\lambda\otimes
IC_{(\lambda,\mu)}.
\end{equation} 

Moreover, it suffices to prove~(\ref{obr}) for $m,l$ divisible by $N$.
In effect, this would imply that the convolution morphism 
$\Gr_{(1,0,\ldots,0)}^{*l}*\bar{\Omega}_{(0^N,0^N)}*\Gr_{(1,0,\ldots,0)}^{*m}  
\to\Gr\times\bVo$ is stratified semismall for {\em any} $m,l\geq0$. 
Indeed, if the direct image of the constant IC sheaf under the above
morphism involved some summands with nontrivial shifts in the 
derived category, the further convolution with $IC_{(1,0,\ldots,0)}$
could not possibly kill the nontrivially shifted summands (due to self-duality
and decomposition theorem), and so they
would persist for some bigger $m,l$ divisible by $N$.

Having established the semismallness for arbitrary $m,l\geq0$, we see
that the semisimple abelian category formed by direct sums of
$IC_{(\lambda,\mu)},\ (\lambda,\mu)\in RB^\sph$ is a bimodule category
over the tensor category formed by direct sums of $IC_\lambda,\
\lambda\in\gS_N^\sph$ (equivalent by Satake isomorphism to
$\on{Rep}(\GL_N)$). To specify such a bimodule category it suffices to
specify the left and right action of the generator $IC_{(1,0,\ldots,0)}$, 
and there
is only one action satisfying~(\ref{obr}) with $m,l$ divisible by $N$:
it necessarily satisfies~(\ref{obr}) for any $m,l$.  

We set $n=m+l$. The advantage of having $n$ divisible by $N$ is that
according to~\ref{Mir}, the (open part of the) orbit closure is
equisingular to certain slice of the $\GL_n$-orbit closure in
$\CN_n\times D$. To describe the convolution diagram in terms of
$\GL_n$ we need to recall a Springer type construction of~\cite{FG}~5.4. 

So $\tY$ is the smooth variety of triples $(u,F_\bullet,v)$ where
$F_\bullet$ is a complete flag in the $n$-dimensional vector space
$D$, and $u$ is a nilpotent endomorphism of $D$ such that $uF_k\subset
F_{k-1}$, and $v\in F_{n-m}$. We have a proper morphism $\pi_{n,m}:\
\tY\to\CN_n\times D$ with the image $\fY\subset\CN_n\times D$ formed
by all the pairs $(u,v)$ such that $\dim\langle
v,uv,u^2v,\ldots\rangle\leq n-m$. It follows from the proof of
Proposition~5.4.1 of {\em loc. cit.} that $\pi_{n,m}$ is a semismall
resolution of singularities, and 

\begin{equation}
\label{fg}
(\pi_{n,m})_*IC(\tY)\simeq
\bigoplus_{|\mu|=m}^{|\lambda|=n-m}\on{L}_\mu\otimes\on{L}_\lambda\otimes
IC_{(\lambda,\mu)}
\end{equation}
where $\on{L}_\mu$ (resp. $\on{L}_\lambda$) is the irreducible
representation of $\gS_m$ (resp. $\gS_{n-m}$) corresponding to the
partition $\mu$ (resp. $\lambda$); furthermore, $IC_{(\lambda,\mu)}$
is the IC sheaf of the $\GL_n$-orbit closure $\overline{(\CN_n\times
  D)}_{(\lambda,\mu)}$ (cf. Theorem~4.5 of~\cite{AH}). 

Recall the nilpotent element $x\in\CN_m$ introduced in~\ref{Mir}, and
the slice $T_x\cap\CN_n$. We will denote
$\pi_{n,m}^{-1}((T_x\cap\CN_n)\times D)$ by $T\tY\subset\tY$.
Recall the open embedding 
$\varphi:\ \CW\times((T_x\cap\CN_n)\times D)\hookrightarrow
\bar{\Omega}_{(n,0,\ldots,0),(0^N)}$ of~\ref{Mir}. 
Let us denote the convolution diagram
$\Gr_{(1,0,\ldots,0)}^{*l}*\bar{\Omega}_{(0^N,0^N)}*\Gr_{(1,0,\ldots,0)}^{*m}$
by $\tilde{\Omega}_{(l,0,\ldots,0),(m,0,\ldots,0)}$ for short; let us
denote its morphism to $\bar{\Omega}_{(n,0,\ldots,0),(0^N)}$ (with the image
$\bar{\Omega}_{(l,0,\ldots,0),(m,0,\ldots,0)}$) by $\varpi_{n,m}$. 
Finally, let us denote the preimage under $\varpi_{n,m}$ of 
$\varphi(\CW\times((T_x\cap\CN_n)\times D))$ by 
$T\tilde{\Omega}_{(l,0,\ldots,0),(m,0,\ldots,0)}$. 
The next lemma follows by comparison of definitions:

\begin{lem}
We have a commutative diagram
$$\begin{CD}
\CW\times T\tY
@>\sim>> T\tilde{\Omega}_{(l,0,\ldots,0),(m,0,\ldots,0)}\\
@VV\id\times\pi_{n,m}V  @V\varpi_{n,m}VV \\
\CW\times ((T_x\cap\CN_n)\times D) @>\varphi>> 
\bar{\Omega}_{(n,0,\ldots,0),(0^N)}
\end{CD}$$
\end{lem}

Since $\on{L}_\lambda=K_\lambda$ by Schur-Weyl duality, the proof of
the theorem is finished. \qed

\begin{Rem}
\label{aha}
Due to Lusztig's construction of~\ref{Lusztig}, Theorem~\ref{summer} implies
Proposition~4.6 of~\cite{AH}.
\end{Rem}

\section{Mirabolic Hall bimodule}
\label{Hall}

\subsection{Recollections} The field $\sk$ is still $\Fq$.
The Hall algebra $\Hall=\Hall_N$ of all finite $\sk[[t]]$-modules
which are direct sums of 
$\leq N$ indecomposable modules is defined as in~\cite{Mac}~II.2. 
It is a quotient
algebra of the ``universal'' Hall algebra $H(\sk[[t]])$ of {\em
  loc. cit.} It has a basis $\{\fu_\lambda\}$ where $\lambda$ runs
through the set $^+\gS_N^\sph$ of partitions with $\leq N$ parts. It
is a free polynomial algebra with generators $\{\fu_{(1^r)},\ 1\leq
r\leq N-1\}$. The structure constants $G^\lambda_{\mu\nu}$ being
polynomial in $q$, we may and will view $\Hall$ as the specialization
under $\bq\mapsto q$ of a $\BZ[\bq,\bq^{-1}]$-algebra
$\bHall$. Extending scalars to $\BZ[\bv,\bv^{-1}]$ we obtain a
$\BZ[\bv,\bv^{-1}]$-algebra $\CHall$.

Let $\Lambda=\Lambda_N$ denote the ring of symmetric polynomials in
the variables $X=(X_1,\ldots,X_N)$ over $\BZ[\bv,\bv^{-1}]$. There is
an isomorphism $\Psi:\ \CHall\iso\Lambda$ sending $\fu_{(1^r)}$ to 
$\bv^{-r(r-1)}e_r$ (elementary symmetric polynomial), and
$\fu_\lambda$ to $\bv^{-2n(\lambda)}P_\lambda(X,\bv^{-2})$ (Chapter~III of
{\em loc. cit.}) where $P_\lambda(X,\bv^{-2})$ 
is the Hall-Littlewood polynomial,
and $n(\lambda)=\sum_{i=1}^N(i-1)\lambda_i$. 
Let us denote by $^+\CH^\sph$ the subalgebra of $\CH^\sph$ spanned by
$\{U_\lambda,\ \lambda\in\ ^+\gS_N^\sph\}$. Then we have a natural
identification of $^+\CH^\sph$ with $\CHall$ sending $U_\lambda$ to
$\fu_\lambda$, and $C_\lambda$ to $\fc_\lambda$. Furthermore,
$\Psi(\fc_\lambda)=(-\bv)^{-(N-1)|\lambda|}s_\lambda$ (Schur polynomial). 

\subsection{The Mirabolic Hall bimodule}
\label{fast}
A finite $\sk[[t]]$-module which is direct sum of $\leq N$
indecomposable modules is the same as a $\sk$-vector space $D$ with a
nilpotent operator $u$ with $\leq N$ Jordan blocks. The isomorphism
classes of pairs $(u,v)$ (where $v\in D$) are numbered by the set
$^+RB^\sph$ of pairs of partitions $(\lambda,\mu)$ with $\leq N$ parts
in $\lambda$ and $\leq N$ parts in $\mu$. We define the structure
constants $G^{(\lambda,\mu)}_{(\lambda',\mu')\nu}$ and
$G^{(\lambda,\mu)}_{\nu(\lambda',\mu')}$ as follows\footnote{The notation
$G^{(\lambda,\mu)}_{(\lambda',\mu')(1^r)}$ and
$G^{(\lambda,\mu)}_{(1^r)(\lambda',\mu')}$ introduced in subsection~\ref{net}
is just a particular case of the present one for $\nu=(1^r)$ as we will
see in Lemma~\ref{bim}.}.
$G^{(\lambda,\mu)}_{\nu(\lambda',\mu')}$ is the number of
$u$-invariant subspaces $D''\subset D$ such that the isomorphism type
of $u|_{D''}$ is given by $\nu$, and the isomorphism type of $(u|_{D/D''},
v\pmod{D''})$ is given by $(\lambda',\mu')$. Furthermore, 
$G^{(\lambda,\mu)}_{(\lambda',\mu')\nu}$ is the number of $u$-invariant
  subspaces $D'\subset D$ containing $v$ such that the isomorphism type
of $(u|_{D'},v)$ is given by $(\lambda',\mu')$, and the isomorphism
type of $u|_{D/D'}$ is given by $\nu$. Note that some similar
quantities were introduced in Proposition~5.8 of~\cite{AH}: in
notations of {\em loc. cit.} we have $g^{\lambda;\mu}_{\theta;\nu}
=\sum_{\lambda'+\mu'=\theta}G^{(\lambda,\mu)}_{(\lambda',\mu')\nu}$.

\begin{lem}
\label{bim}
For any $^+RB^\sph\ni(\lambda,\mu),(\lambda',\mu'),\
1\leq r\leq N-1$, the structure constants 
$G^{(\lambda,\mu)}_{(1^r)(\lambda',\mu')}$ and
$G^{(\lambda,\mu)}_{(\lambda',\mu')(1^r)}$ are given by the
formulas~(\ref{demos}) and~(\ref{zhut'}). 
\end{lem}

\begin{proof} 
Was given in subsection~\ref{net}.
\end{proof}


We define the Mirabolic Hall bimodule $\Mall$ over $\Hall$ to have a
$\BZ$-basis $\{\fu_{(\lambda,\mu)},\ (\lambda,\mu)\in\ ^+RB^\sph\}$
and the structure constants
$$\fu_\nu\fu_{(\lambda',\mu')}=\sum_{(\lambda,\mu)\in\ ^+RB^\sph} 
G^{(\lambda,\mu)}_{\nu(\lambda',\mu')}\fu_{(\lambda,\mu)},\
\fu_{(\lambda',\mu')}\fu_\nu=\sum_{(\lambda,\mu)\in\ ^+RB^\sph} 
G^{(\lambda,\mu)}_{(\lambda',\mu')\nu}\fu_{(\lambda,\mu)}$$
The structure constants $G^{(\lambda,\mu)}_{(\lambda',\mu')(1^r)}$ and
$G^{(\lambda,\mu)}_{(1^r)(\lambda',\mu')}$ for the generators of
  $\Hall$ being polynomial in $q$, we may and will view $\Mall$ as the
  specialization under $\bq\mapsto q$ of a
  $\BZ[\bq,\bq^{-1}]$-bimodule $\bMall$ over the
  $\BZ[\bq,\bq^{-1}]$-algebra $\bHall$. Extending scalars to
  $\BZ[\bv,\bv^{-1}]$ we obtain a $\BZ[\bv,\bv^{-1}]$-bimodule  
$\CMall$ over the  $\BZ[\bv,\bv^{-1}]$-algebra $\CHall$.
Let us denote by $^+\CR^\sph$ the $^+\CH^\sph$-subbimodule of
$\CR^\sph$ spanned by $\{U_{(\lambda,\mu)},\ (\lambda,\mu)\in\
^+RB^\sph\}$. Then we have a natural identification of $^+\CR^\sph$
with $\CMall$ sending $U_{(\lambda,\mu)}$ to $\fu_{(\lambda,\mu)}$.
For $(\lambda,\mu)\in\ ^+RB^\sph$ we set $^+C_{(\lambda,\mu)}:=
\sum_{^+RB^\sph\ni(\lambda',\mu')\leq(\lambda,\mu)} 
\Pi_{(\lambda',\mu'),(\lambda,\mu)}H_{(\lambda',\mu')}$
(notation introduced after Proposition~\ref{sph KL basis}). We define
$\fc_{(\lambda,\mu)}\in\CMall$ as the element corresponding to
$^+C_{(\lambda,\mu)}$ under the above identification. 

Theorem~\ref{summer} admits the following corollary:

\begin{cor}
\label{fall} For any $\lambda,\mu\in\ ^+\gS^\sph_N$ we have 
$\fc_\lambda\fc_{(0^N,0^N)}\fc_\mu=\fc_{(\lambda,\mu)}$.
\end{cor}

Hence there is a unique isomorphism $\Psi:\
\CMall\iso\Lambda\otimes\Lambda$ of $\CHall\simeq\Lambda$-bimodules
sending $\fc_{(\lambda,\mu)}$ to
$(-\bv)^{-(N-1)(|\lambda|+|\mu|)}s_\lambda\otimes s_\mu$. 
We define $$\Lambda\otimes\Lambda\ni P_{(\lambda,\mu)}(X,Y,\bv^{-1}):= 
(-\bv)^{2n(\lambda)+2n(\mu)+|\mu|}\Psi(\fu_{(\lambda,\mu)})$$
(mirabolic Hall-Littlewood polynomials).

Thus the polynomials $\Pi_{(\lambda',\mu'),(\lambda,\mu)}$ are the
matrix coefficients of the transition matrix from the basis 
$\{P_{(\lambda,\mu)}(X,Y,\bv^{-1})\}$ to the basis
$\{s_\lambda(X)s_\mu(Y)\}$ of $\Lambda\otimes\Lambda$. 
It follows from Theorem~5.2 of~\cite{AH} that the mirabolic
Hall-Littlewood polynomial $P_{(\lambda,\mu)}(X,Y,\bv^{-1})$ coincides
with Shoji's Hall-Littlewood function
$P^\pm_{(\lambda,\mu)}(X,Y,\bv^{-1})$ (see section~2.5 and 
Theorem~2.8 of~\cite{Sh}).

\section{Frobenius traces in mirabolic character sheaves}

\subsection{Unipotent mirabolic character sheaves}
\label{unip}
Recall the construction of certain mirabolic character sheaves
in~\cite{FG}~5.4. So $\tX$ is the smooth variety of triples
$(g,F_\bullet,v)$ where $F_\bullet$ is a complete flag in an 
$n$-dimensional vector
space $D$, and $v\in F_m$, and $g$ is an invertible linear
transformation of $D$ preserving $F_\bullet$. We have a proper
morphism $\pi_{n,m}:\ \tX\to\GL_n\times D$ with the image
$\fX\subset\GL_n\times D$ formed by all the pairs $(g,v)$ such that
$\dim\langle v,gv,g^2v,\ldots\rangle\leq n-m$. According to
Corollary~5.4.2 of {\em loc. cit.}, we have 

\begin{equation}
\label{last}
(\pi_{n,m})_*IC(\tX)\simeq\bigoplus^{|\lambda|=n-m}_{|\mu|=m}
\on{L}_\mu\otimes\on{L}_\lambda\otimes\CF_{\lambda,\mu}
\end{equation}
for certain irreducible perverse mirabolic character sheaves
$\CF_{\lambda,\mu}$ on $\GL_n\times D$. 

Following~[AH], we set $b(\lambda,\mu):=2n(\lambda)+2n(\mu)+|\mu|$, 
so that $b(\lambda',\mu')-b(\lambda,\mu)$ equals the codimension of 
$\Omega_{(\lambda',\mu')}$ in $\bar{\Omega}_{(\lambda,\mu)}$, and
the codimension of 
$(\CN_n\times D)_{(\lambda',\mu')}$ in 
$\overline{(\CN_n\times D)}_{(\lambda,\mu)}$.

\begin{thm}
\label{past}
Let $(u,v)\in(\CN_n\times D)_{(\lambda',\mu')}(\Fq)$.
The trace of Frobenius automorphism of the stalk of
$\CF_{\lambda,\mu}$ at $(u,v)$ equals
${\sqrt{q}}^{b(\lambda',\mu')-b(\lambda,\mu)}
\Pi_{(\lambda',\mu'),(\lambda,\mu)}(\sqrt{q}) (see~(\ref{ginz}))$.
\end{thm}

\begin{proof} 
We identify the nilpotent cone $\CN_n$ and the variety of unipotent 
elements of $\GL_n$ by adding the identity matrix, 
so that we may view $\CN_n\subset\GL_n$. Then
$\fX\cap(\CN_n\times D)=\fY$, and $\pi_{n,m}^{-1}(\fX\cap(\CN_n\times
D))=\tY$ (notations of the proof of
Theorem~\ref{summer}). Comparing~(\ref{last}) with~(\ref{fg}), we see
that $\CF_{\lambda,\mu}|_{\CN_n\times D}\simeq IC_{(\lambda,\mu)}$.  
Hence the trace of Frobenius in the stalk of
$\CF_{\lambda,\mu}$ at $(u,v)$ equals the trace of Frobenius in the
stalk of $IC_{(\lambda,\mu)}$ at $(u,v)$. The latter is equal to the matrix
coefficient of the transition matrix from the basis 
$\{j_!\Ql_{(\CN_n\times D)_{(\lambda',\mu')}}[n^2-b(\lambda',\mu')]\}$ 
to the basis
$\{j_{!*}\Ql_{(\CN_n\times D)_{(\lambda,\mu)}}[n^2-b(\lambda,\mu)]\}$.
And the latter by construction, up to the factor of
${\sqrt{q}}^{b(\lambda',\mu')-b(\lambda,\mu)}$, is equal to 
$\Pi_{(\lambda',\mu'),(\lambda,\mu)}(\sqrt{q})$. \end{proof}

\subsection{$\Gm$-equivariant mirabolic character sheaves}
More generally, we recall the notion~\cite{FG2} of mirabolic character sheaves 
equivariant with respect to the dilation action of $\Gm$ on $D$. 
Let $\CB$ be the flag variety of $\GL(D)$, let $\tB$ be the base
affine space of $\GL(D)$, so that $\tB\to\CB$ is a
$\GL(D)$-equivariant $H$-torsor, where $H$ is the abstract Cartan of
$\GL(D)$. Let $\CY$ be the quotient of $\tB\times\tB$ modulo the
diagonal action of $H$; it is called the horocycle space of
$\GL(D)$. Clearly, $\CY$ is an $H$-torsor over $\CB\times\CB$ with
respect to the right action $(\tilde{x}_1,\tilde{x}_2)\cdot
h:=(\tilde{x_1}\cdot h,\tilde{x}_2)$. We consider the following
diagram of $\GL(D)$-varieties and $\GL(D)\times\Gm$-equivariant maps:
$$\GL(D)\times D\stackrel{pr}{\longleftarrow}\GL(D)\times\CB\times D
\stackrel{f}{\longrightarrow}\CY\times D.$$
In this diagram, the map $pr$ is given by $pr(g,x,v):=(g,v)$. To
define the map $f$, we think of $\CB$ as $\tB/H$, and for a
representative ${\tilde x}\in\tB$ of $x\in\CB$ we set
$f(g,x,v):=(g{\tilde x},{\tilde x},gv)$. The group $\GL(D)$ acts
diagonally on all the product spaces in the above diagram, and acts on
itself by conjugation. The group $\Gm$ acts by dilations on $D$.

The functor $\Char$ from the constructible
derived category of $l$-adic sheaves on $\CY\times D$ to the
constructible derived category of $l$-adic sheaves on $\GL(D)\times D$
is defined as $\Char:=pr_*f^![-\dim\CB]$. Now let $\CF$ be a
$\GL(D)\times\Gm$-equivariant, $H$-monodromic perverse sheaf on
$\CY\times D$. The irreducible perverse constituents of $\Char\CF$ are
called $\Gm$-equivariant mirabolic character sheaves on $\GL(D)\times
D$. We define a $\Gm$-equivariant mirabolic character sheaf as a
direct sum of the above constituents for various $\CF$ as above.
The semisimple abelian category of $\Gm$-equivariant mirabolic
character sheaves will be denoted $\MC(\GL(D)\times D)$.
Clearly, this definition is a direct analogue of Lusztig's definition
of character sheaves. The semisimple abelian category of character
sheaves on $\GL(D)$ will be denoted $\calC(\GL(D))$.

\subsection{Left and right induction}
\label{lr}
Following Lusztig's construction of {\em induction} of character
sheaves, we define the left and right action of Lusztig's character
sheaves on the mirabolic character sheaves (for varying $D$).
To this end it will be notationally more convenient to consider 
$\MC(\GL(D)\times D)$ (resp. $\calC(\GL(D))$) as a category of perverse
sheaves on the quotient stack $\GL(D)\backslash(\GL(D)\times D)$
(resp. $\GL(D)\backslash\GL(D)$). Let $m\leq n=\dim(D)$, and let $V$
be an $n-m$-dimensional $\sk$-vector space, and let $W$ be an
$m$-dimensional $\sk$-vector space. 
We have the following diagrams:

\begin{equation}
\label{left}
\GL(D)\backslash(\GL(D)\times D)\stackrel{p}{\longleftarrow}
A\stackrel{q}{\longrightarrow}\GL(V)\backslash\GL(V)\ \times\
\GL(W)\backslash(\GL(W)\times W),
\end{equation}

\begin{equation}
\label{right}
\GL(D)\backslash(\GL(D)\times D)\stackrel{d}{\longleftarrow}
B\stackrel{b}{\longrightarrow}\GL(V)\backslash(\GL(V)\times V)\ \times\
\GL(W)\backslash\GL(W).
\end{equation}
Here $A$ is the quotient stack of $\tilde A$ by the action of
$\GL(D)$, and $${\tilde A}:=\{(g\in\GL(D), F\subset D, v\in D)\
\on{such}\ \on{that}\ \dim F=n-m,\ \on{and}\ gF=F\},$$
and, $p$ forgets $F$, and $q$ sends $(g,F,v)$ to
$g|_F;(g|_{D/F}, v\pmod{F})$ (under an arbitrary identification
$V\simeq F,\ W\simeq D/F$). Note that $p$ is proper, and $q$ is smooth
of relative dimension $n-m$.

Furthermore,  $B$ is the quotient stack of $\tilde B$ by the action of
$\GL(D)$, and $${\tilde B}:=\{(g\in\GL(D), F\subset D, v\in F)\
\on{such}\ \on{that}\ \dim F=n-m,\ \on{and}\ gF=F\},$$
and, $d$ forgets $F$, and $b$ sends $(g,F,v)$ to
$(g|_F,v);g|_{D/F})$ (under an arbitrary identification
$V\simeq F,\ W\simeq D/F$). Note that $d$ is proper, and $b$ is smooth
of relative dimension 0.

Finally, for a character sheaf $\CG\in\calC(\GL(V))$ and a mirabolic
character sheaf $\CF\in\MC(\GL(W)\times W)$ we define the {\em left}
convolution $\CG*\CF:=p_!q^*(\CG\boxtimes\CF)[n-m]$. Similarly,
for a character sheaf $\CG'\in\calC(\GL(W))$ and a mirabolic
character sheaf $\CF'\in\MC(\GL(V)\times V)$ we define the {\em right}
convolution $\CF'*\CG':=d_!b^*(\CF'\boxtimes\CG')$.

Note that the definition of convolution works in the extreme cases 
$m=0$ or $n-m=0$ as well: if $\dim V=0$, then $\GL(V)$ is just the
trivial group.
The following proposition is proved like Proposition~4.8.b) in~\cite{L3}.

\begin{prop}
\label{refer}
Both $\CG*\CF$ and $\CF'*\CG'$ are $\Gm$-equivariant mirabolic
character sheaves on $\GL(D)\times D$.
\end{prop}

We denote by $\Ql$ the unique $\Gm$-equivariant mirabolic character
sheaf on $\GL(D)\times D$ for $\dim(D)=0$.

\begin{prop}
\label{small}
Let $\CG\in\calC(\GL(V))$, and $\CG'\in\calC(\GL(W))$ be {\em
  irreducible} character sheaves. Then $\CG*\Ql*\CG'$ is irreducible.
\end{prop}

\begin{proof} Let $\dim(D)=n,\ \dim(W)=m,\ \dim(V)=n-m$. Recall the
diagram~(\ref{right}), and denote by $r:\
\GL(V)\backslash(\GL(V)\times V)\to\GL(V)\backslash\GL(V)$ the natural
projection (forgetting vector $v$). Then
$\CG*\Ql*\CG'=d_!b^*(r^*\CG\boxtimes\CG'[n-m])$.
The sheaf $b^*(r^*\CG\boxtimes\CG'[n-m])$ is irreducible perverse on
$B$; more precisely, it is the intermediate extension of a local
system on an open part of $B$. 
The image of proper morphism $d$ coincides with
$\GL(D)\backslash\fX$ (notations of~\ref{unip}), and 
$d:\ B\to\GL(D)\backslash\fX$ is generically isomorphism:
$F$ is reconstructed as $F=\langle v,gv,g^2v,\ldots\rangle$.
Finally, the arguments absolutely similar to the proof of
Proposition~4.5 of~\cite{L2.5} prove that $d$ is stratified small.
This implies that $d_!b^*(r^*\CG\boxtimes\CG'[n-m])$ is
irreducible. \end{proof}

\begin{conj}
\label{big}
Any irreducible $\Gm$-equivariant mirabolic character sheaf on
$\GL(D)\times D$ is isomorphic to $\CG*\Ql*\CG'$ for some 
$\CG\in\calC(\GL(V))$, and $\CG'\in\calC(\GL(W))$ where
$\dim(V)+\dim(W)=\dim(D)$. 
\end{conj}

\subsection{Mirabolic Green bimodule}
\label{medium}
Once again $\sk=\Fq$.
We will freely use the notation of Chapter~IV of~\cite{Mac}.
In particular, $\Phi$ is the set of Frobenius orbits in
$\overline{\Fq}{}^\times$, or equivalently, the set of irreducible
monic polynomials in $\Fq[t]$ with the exception of $f=t$. We consider
the set of isomorphism classes $(D,g,v)$ where $D$ is a $\sk$-vector
space, $v\in D$, and $g$ is an invertible linear operator $D\to D$. 
Similarly to section~2 of {\em loc. cit.} we identify this set with
the set of finitely supported functions 
$(\blambda,\bmu):\ \Phi\to\CP\times\CP$ 
to the set of pairs of partitions. Note that
$\dim(D)=|(\blambda,\bmu)|:=\sum_{f\in\Phi}\deg(f)(|\blambda(f)|+|\bmu(f)|)$. 
Let $c_{(\blambda,\bmu)}\subset\GL(D)\times D$ be the corresponding
$\GL(D)$-orbit, and let $\bpi_{(\blambda,\bmu)}$ be its characteristic
function. Let $\MA$ be a $\overline{\mathbb Q}_l$-vector space with
the basis $\{\bpi_{(\blambda,\bmu)}\}$. It is evidently isomorphic to 
$\bigoplus_{n\geq0}\overline{\mathbb
  Q}_l(\GL(\sk^n)\times\sk^n)^{\GL(\sk^n)}$. 

Recall the Green algebra $\CA=\bigoplus_{n\geq0}\CA_n$ of class functions
on the groups $\GL_n(\Fq)$ (see section~3 of {\em loc. cit.}; the
multiplication is given by parabolic induction) with the
basis $\{\bpi_\bmu\}$ of characteristic functions of conjugacy
classes. The construction of~\ref{lr} equips $\MA$ with a structure of an
$\CA$-bimodule. It is easily seen to be a free bimodule of rank 1 with
a generator $\bpi_{(\bzero,\bzero)}$
given by the zero function (taking the value of zero
bipartition on any $f\in\Phi$). The structure constants are as
follows (the proof is similar to~(3.1) of {\em loc. cit.}).

\begin{equation}
\label{mgreen}
\bpi_\bnu\bpi_{(\blambda',\bmu')}=\sum_{(\blambda,\bmu)} 
\on{g}^{(\blambda,\bmu)}_{\bnu(\blambda',\bmu')}\bpi_{(\blambda,\bmu)},\
\bpi_{(\blambda',\bmu')}\bpi_\bnu=\sum_{(\blambda,\bmu)} 
\on{g}^{(\blambda,\bmu)}_{(\blambda',\bmu')\bnu}\bpi_{(\blambda,\bmu)},
\end{equation}
where 
\begin{equation}
\label{33}
\on{g}^{(\blambda,\bmu)}_{\bnu(\blambda',\bmu')}=
\prod_{f\in\Phi}G^{(\blambda(f),\bmu(f))}_{\bnu(f)(\blambda'(f),\bmu'(f))}
(q^{\deg(f)}),\ 
\on{g}^{(\blambda,\bmu)}_{(\blambda',\bmu')\bnu}=
\prod_{f\in\Phi}G^{(\blambda(f),\bmu(f))}_{(\blambda'(f),\bmu'(f))\bnu(f)}
(q^{\deg(f)}).
\end{equation}

Now recall another basis $\{S_\bfeta\}$ of $\CA$ (see section~4 of
{\em loc. cit.}), numbered by the finitely supported functions from
$\Theta$ to $\CP$. Here $\Theta$ is the set of Frobenius orbits on the
direct limit $L$ of character groups $({\mathbb
  F}_{q^n}^\times)^\vee$. This is the basis of irreducible
characters. According to Lusztig, for $|\bfeta|=m$, the function
$S_\bfeta$ is the Frobenius trace function of an irreducible Weil character
sheaf $\CS_\bfeta$ on $\GL_m$. Due to Proposition~\ref{small},
for $|\bfeta|+|\bfeta'|=n$, the function
$S_{\bfeta'}\bpi_{(\bzero,\bzero)}S_\bfeta$ is the Frobenius trace
function of an irreducible $\Gm$-equivariant Weil mirabolic character sheaf
$\CS_{\bfeta'}*\Ql*\CS_\bfeta$
on $\GL(D)\times D,\ \dim(D)=n$. We know that the set of functions 
$\{S_{\bfeta'}\bpi_{(\bzero,\bzero)}S_\bfeta\}$ forms a basis of the
mirabolic Green bimodule $\MA$. Hence, if Conjecture~\ref{big} holds
true, then the set of Frobenius trace functions of irreducible
$\Gm$-equivariant Weil mirabolic character sheaves forms a basis of
$\MA$. This would be a positive answer to a question of G.~Lusztig.

\end{document}